\documentclass[11pt,leqno]{amsart}

\usepackage[a4paper, right=20mm, left=20mm, top=25mm]{geometry}

\usepackage{microtype}

\usepackage{amsthm}
\usepackage[english]{babel}
\usepackage[utf8]{inputenc}
\usepackage{csquotes}
\usepackage{hyperref}
\usepackage[foot]{amsaddr}

\usepackage{amsmath}
\usepackage{amssymb}
\usepackage{mathtools}
\usepackage{enumitem}

\usepackage{tikz}
\usetikzlibrary{matrix,arrows,positioning}

\numberwithin{equation}{section}
\numberwithin{figure}{section}

\theoremstyle{plain}
\newtheorem{Theorem}{Theorem}[section]
\newtheorem*{Theorem*}{Theorem}
\newtheorem{Lemma}[Theorem]{Lemma}
\newtheorem*{Lemma*}{Lemma}
\newtheorem{Proposition}[Theorem]{Proposition}
\newtheorem*{Proposition*}{Proposition}
\newtheorem{Corollary}[Theorem]{Corollary}

\theoremstyle{definition}
\newtheorem{Definition}[Theorem]{Definition}
\newtheorem*{Definition*}{Definition}
\newtheorem{Remark}[Theorem]{Remark}

\newcommand{\C}{\mathbb{C}}
\newcommand{\R}{\mathbb{R}}
\newcommand{\Q}{\mathbb{Q}}
\newcommand{\Z}{\mathbb{Z}}
\newcommand{\kk}{\Bbbk}

\newcommand{\Hom}{\mathrm{Hom}}
\newcommand{\End}{\mathrm{End}}
\newcommand{\Ext}{\mathrm{Ext}}
\newcommand{\id}{\mathrm{id}}

\newcommand{\res}{\mathrm{res}}
\newcommand{\Cdot}{\boldsymbol{\cdot}}
\newcommand{\Stab}{\mathrm{Stab}}
\newcommand{\Rep}{\mathrm{Rep}}
\newcommand{\Tilt}{\mathrm{Tilt}}
\newcommand{\GL}{\mathrm{GL}}
\newcommand{\im}{\mathrm{im}}
\newcommand{\pr}{\mathrm{pr}}
\newcommand{\soc}{\mathrm{soc}}
\newcommand{\head}{\mathrm{head}}
\newcommand{\rad}{\mathrm{rad}}

\providecommand{\abs}[1]{\lvert#1\rvert}

\title[Minimal tilting complexes]{On minimal tilting complexes\\in highest weight categories}
\author{Jonathan Gruber}
\address{Department of Mathematics, National University of Singapore, Singapore}
\email{jgruber@nus.edu.sg}
\date{\today}

\begin{document}

\begin{abstract}
	We explain the construction of \emph{minimal tilting complexes} for objects of highest weight categories and we study in detail the minimal tilting complexes for standard objects and simple objects.
	For certain categories of representations of complex simple Lie algebras, affine Kac-Moody algebras and quantum groups at roots of unity, we relate the multiplicities of indecomposable tilting objects appearing in the terms of these complexes to the coefficients of Kazhdan-Lusztig polynomials.
\end{abstract}

\maketitle

\section*{Introduction}

Highest weight categories are ubiquitous in Lie theoretic representation theory, and the study of tilting objects in highest weight categories has received a lot of attention in recent years.
One key problem is the determination of the multiplicities of standard objects in $\Delta$-filtrations of indecomposable tilting objects \cite{SoergelCharakterformeln,RicheWilliamsonTiltingpCanonical}.
In this article, we take a somewhat opposite approach and try to understand standard objects and simple objects in a highest weight category in terms of (minimal) complexes of tilting objects.

Let $\kk$ be a field, let $\mathcal{C}$ be a $\kk$-linear highest weight category with (lower finite) weight poset $(\Lambda,\leq)$ and write $\Tilt(\mathcal{C})$ for the full subcategory of tilting objects in $\mathcal{C}$.
For $\lambda \in \Lambda$, we write $L_\lambda$, $\Delta_\lambda$, $\nabla_\lambda$ and $T_\lambda$ for the corresponding simple object, standard object, costandard object and indecomposable tilting object of $\mathcal{C}$, respectively.
It is well known that the canonical functor
\[ \mathfrak{T} \colon K^b\big( \Tilt(\mathcal{C}) \big) \longrightarrow D^b(\mathcal{C}) \]
from the bounded homotopy category of $\Tilt(\mathcal{C})$ to the bounded derived category of $\mathcal{C}$ is an equivalence of triangulated categories, and this allows us to associate every bounded complex $X$ in $\mathcal{C}$ with a unique homotopy class of bounded complexes in $\Tilt(\mathcal{C})$ that are isomorphic to $X$ in $D^b(\mathcal{C})$.
Furthermore, every homotopy class in $K^b\big( \Tilt(\mathcal{C}) \big)$ contains a \emph{minimal complex} (see Subsection \ref{subsec:minimalcomplexes} below) which is unique up to isomorphism of complexes, so we can associate $X$ with a unique (up to isomorphism) \emph{minimal tilting complex} $C_\mathrm{min}(X)$ that is minimal and isomorphic to $X$ in $D^b(\mathcal{C})$.
Given an object $M$ of $\mathcal{C}$, we can view $M$ as an object of $D^b(\mathcal{C})$ (as the complex whose only non-zero term is $M$, in homological degree zero) and thus define the minimal tilting complex $C_\mathrm{min}(M)$ of $M$.
Minimal tilting complexes are well-behaved with respect to direct sums and short exact sequences in $\mathcal{C}$ (see Lemmas \ref{lem:minimalcomplex} and \ref{lem:triangle}),
and the minimal tilting complex of an object $M$ of $\mathcal{C}$ encodes some important information, like the $\nabla$-filtration dimension $\dim_\nabla(M)$ and the $\Delta$-filtration dimension $\dim_\Delta(M)$ of $M$ (see Lemma \ref{lem:minimaltiltingcomplexgfdwfd}).
After establishing these general properties, we will focus on the minimal tilting complexes $C_\mathrm{min}(\Delta_\lambda)$ and $C_\mathrm{min}(L_\lambda)$, for $\lambda \in \Lambda$.
One important tool for studying these complexes is the Ringel dual of $\mathcal{C}$ and the Ringel duality functors, which we briefly recall below.

Suppose that the weight poset $\Lambda$ is finite and let $T= \bigoplus_{\lambda\in\Lambda} T_\lambda$ be the \emph{characteristic tilting object}.
The \emph{Ringel dual} $\mathcal{C}^\prime$ of $\mathcal{C}$ is defined as the category of finite dimensional left $\End_\mathcal{C}(T)^\mathrm{op}$-modules and the \emph{Ringel duality functors} are $R_\mathcal{C}^\Delta = \Hom_\mathcal{C}(T,-)$ and $R_\mathcal{C}^\nabla = \Hom_\mathcal{C}(-,T)^*$.
Then $\mathcal{C}^\prime$ is a highest weight category with weight poset $(\Lambda,\leq^\mathrm{op})$, standard objects $\Delta^\prime_\lambda = R_\mathcal{C}^\Delta(\nabla_\lambda)$, costandard objects $\nabla^\prime_\lambda = R_\mathcal{C}^\nabla(\Delta_\lambda)$ and simple objects $L^\prime_\lambda = \head \, \Delta^\prime_\lambda = \soc \, \nabla^\prime_\lambda$.
For an arbitrary (not necessarily finite) poset $(\Lambda,\leq)$ and a finite order ideal $\Gamma \subseteq \Lambda$, the Serre subcategory $\mathcal{C}_\Gamma$ of $\mathcal{C}$ generated by the simple objects $L_\lambda$, with $\lambda \in \Gamma$, is a highest weight category and we may consider its Ringel dual $\mathcal{C}_\Gamma^\prime$.
By abuse of notation, we still denote the standard objects, the costandard objects and the simple objects of $\mathcal{C}_\Gamma^\prime$ by $\Delta^\prime_\lambda$, $\nabla^\prime_\lambda$ and $L^\prime_\lambda$, respectively, when the choice of $\Gamma$ is clear from the context.
The starting point for our investigation of the minimal tilting complexes of standard objects and simple objects is the following lemma (which may already be known to experts).
It relates the multiplicities of indecomposable tilting objects appearing as direct summands in the terms of minimal tilting complexes of standard objects in $\mathcal{C}$ with dimensions of $\Ext$-groups in the Ringel dual; see Lemma \ref{lem:minimaltiltingcomplexstandard}.

\begin{Lemma*}
	Let $\lambda \in \Lambda$ and write $C_\mathrm{min}(\Delta_\lambda) = (T_\bullet,d_\bullet)$.
	For any $i \in \Z$ and $\mu \in \Lambda$ and for any finite order ideal $\Gamma \subseteq \Lambda$ with $\lambda \in \Gamma$, we have
	\[ [ T_i : T_\mu ]_\oplus = \begin{cases} \dim \Ext_{\mathcal{C}_\Gamma^\prime}^i(L_\mu^\prime,\nabla_\lambda^\prime) & \text{if } i \geq 0 \text{ and } \mu \leq \lambda , \\ 0 & \text{otherwise} . \end{cases} \]
\end{Lemma*}

We also establish upper bounds for the multiplicities of indecomposable tilting objects appearing in the terms of minimal tilting complexes of simple objects in $\mathcal{C}$; see Proposition \ref{prop:minimaltiltingcomplexsimple}.

\begin{Proposition*}
	For $\lambda \in \Lambda$, there exists a bounded complex $C_\lambda = (T_\bullet,d_\bullet)$ of tilting objects in $\mathcal{C}$ with 
	\[ L_\lambda \cong C_\lambda \qquad \text{in } D^b(\mathcal{C}) \]
	and such that the terms of $C_\lambda$ satisfy $[ T_i : T_\mu ]_\oplus = 0$ for all $i \in \Z$ and $\mu \in \Lambda$ with $\mu \nleq \lambda$ and
	\[ [ T_i : T_\mu ]_\oplus = \sum_{\nu \leq \lambda} ~ \sum_{k-j=i} \dim \Ext_\mathcal{C}^j(L_\lambda,\nabla_\nu) \cdot \dim \Ext_{\mathcal{C}_\Gamma^\prime}^k(L_\mu^\prime,\nabla_\nu^\prime) \]
	for all $i \in \Z$ and $\mu \in \Lambda$ with $\mu \leq \lambda$ and for any finite order ideal $\Gamma \subseteq \Lambda$ with $\lambda \in \Gamma$.
	Furthermore, the minimal tilting complex $C_\mathrm{min}(L_\lambda)$ is a direct summand of $C_\lambda$ in the category of complexes in $\Tilt(\mathcal{C})$.
\end{Proposition*}

In many specific cases, the complex $C_\lambda$ from the preceding proposition actually coincides with the minimal tilting complex $C_\mathrm{min}(L_\lambda)$.
This is the case, for instance, when $\mathcal{C}$ is a block of the BGG category $\mathcal{O}$ for a complex simple Lie algebra, a (suitable truncation of a) block of category $\mathcal{O}$ for an affine Kac-Moody algebra, or a block of a quantum group at a root of unity.
More specifically, in all of these cases, we can identify the dimensions of the $\Ext$-groups that appear in the preceding lemma and proposition with the coefficients of certain Kazhdan-Lusztig polynomials, and parity considerations then imply that the complex $C_\lambda$ in the proposition is minimal.
We state below special cases of our results for blocks of category $\mathcal{O}$ for complex semisimple Lie algebras and for blocks of quantum groups at roots of unity, but we first need to introduce some more notation.

Let $\mathfrak{g}$ be a complex semisimple Lie algebra with Borel subalgebra $\mathfrak{b}$ and Cartan subalgebra $\mathfrak{h} \subseteq \mathfrak{b}$.
For $\lambda \in \mathfrak{h}^*$, let $\Delta_\lambda$, $\nabla_\lambda$, $L_\lambda$ and $T_\lambda$ be the Verma module, the dual Verma module, the simple $\mathfrak{g}$-module and the indecomposable tilting $\mathfrak{g}$-module of highest weight $\lambda$, respectively, in the BGG category $\mathcal{O}$ of finitely generated $\mathfrak{g}$-modules that are locally $\mathfrak{b}$-finite and admit a weight space decomposition with respect to $\mathfrak{h}$.
We write $W \subseteq \GL(\mathfrak{h}^*)$ for the Weyl group of $\mathfrak{g}$ and $S \subseteq W$ for the set of simple reflections with respect to $\mathfrak{b}$, and we denote by $Q \subseteq X \subseteq \mathfrak{h}^*$ the root lattice and the weight lattice of $\mathfrak{g}$, respectively.
Let $\rho$ be the half-sum of all positive roots of $\mathfrak{g}$ with respect to $\mathfrak{b}$ and consider the dot action of $W$ on $\mathfrak{h}^*$ given by $w \Cdot \lambda = w( \lambda + \rho ) - \rho$ for $w \in W$ and $\lambda \in \mathfrak{h}^*$.
Our normalization for Kazhdan-Lusztig polynomials is as in \cite{SoergelKL}.

\begin{Theorem*}
	Let $x,y \in W$ and let $\lambda \in X$ be a strictly anti-dominant weight whose stabilizer with respect to the dot action of $W$ on $\mathfrak{h}^*$ is trivial.
	Then we have
	\[ \sum_{i \in \Z} [ C_\mathrm{min}(\Delta_{x\Cdot\lambda})_i : T_{y\Cdot\lambda} ]_\oplus \cdot v^i = h^{x,y} \]
	and
	\[ \sum_{i \in \Z} [ C_\mathrm{min}(L_{x\Cdot\lambda})_i : T_{y\Cdot\lambda} ]_\oplus \cdot v^i = \sum_{z \in W } \overline{h_{z,x}} \cdot h^{z,y} , \]
	where $h_{y,x}$ is the Kazhdan-Lusztig polynomial and $h^{x,y}$ is the inverse Kazhdan-Lusztig polynomial corresponding to $x$ and $y$, and where we write $\overline{\phantom{A}} \colon \Z[v^{\pm 1}] \to \Z[v^{\pm 1}]$ for the unique ring homomorphism with $v \mapsto v^{-1}$.
\end{Theorem*}

In Subsection \ref{subsec:categoryO}, we also discuss the case of simple $\mathfrak{g}$-modules and Verma modules of singular highest weight (i.e.\ where the stabilizer of the highest weight with respect to the dot action of $W$ is non-trivial) and of parabolic versions of category $\mathcal{O}$; see Theorem \ref{thm:minimaltiltingcomplexescategoryO}.
Moreover, we establish similar results for blocks of category $\mathcal{O}$ for an affine Kac-Moody algebra in Subsection \ref{subsec:KacMoody}.

Now assume that $\mathfrak{g}$ is a simple Lie algebra, fix a positive integer $\ell$ and let $U_\zeta = U_\zeta(\mathfrak{g})$ be the quantum group corresponding to $\mathfrak{g}$, specialized at an $\ell$-th root of unity $\zeta \in \C$, as defined for example in \cite[Section 7]{TanisakiCharacterFormulas}.
In order to avoid technicalities, we further assume here that $\ell$ is odd and that $\ell$ is not divisible by $3$ if $\mathfrak{g}$ is of type $\mathrm{G}_2$.
(We consider more general values of $\ell$ in Subsection \ref{subsec:quantum}.)
For a dominant weight $\lambda \in X$, let $L^\zeta_\lambda$, $\Delta^\zeta_\lambda$, $\nabla^\zeta_\lambda$ and $T^\zeta_\lambda$ be the simple $U_\zeta$-module, the Weyl module, the induced module and the indecomposable tilting $U_\zeta$-module of highest weight $\lambda$, respectively, in the category $\Rep(U_\zeta)$ of finite-dimensional $U_\zeta$-modules that admit a weight space decomposition.
The affine Weyl group of $U_\zeta$ is the semidirect product $W_\ell = Q \rtimes W$, and we consider its $\ell$-dilated dot action on $X$, which is given by
\[ t_\gamma w \Cdot_\ell \lambda = w( \lambda + \rho ) - \rho + \ell \gamma , \]
for $\gamma \in Q$, $w \in W$ and $\lambda \in X$, where we write $\gamma \mapsto t_\gamma$ for the canonical embedding $Q \to W_\ell$.
The \emph{fundamental dominant $\ell$-alcove} is a canonical subset $C_\ell$ of $X_\R = X \otimes_\Z \R$ whose closure is a fundamental domain for the $\ell$-dilated dot action of $W_\ell$ on $X_\R$ (see Subsection \ref{subsec:quantum}), and $W_\ell$ is a Coxeter group with simple reflections given by the reflections in the walls of $C_\ell$.
Furthermore, the finite Weyl group $W$ is a parabolic subgroup of $W_\ell$ and we write $\prescript{S}{}{W_\ell}$ for the set of elements $x \in W_\ell$ that have minimal length in the coset $W x$.
Finally, according to \cite{KL12,KL34}, there is a \emph{Kazhdan-Lusztig functor} $F_\ell$ from a certain category of representations of an affine Lie algebra to $\Rep(U_\zeta)$, and $F_\ell$ is an equivalence of categories when $\ell$ is sufficiently large.
(See Subsection \ref{subsec:quantum} for more details.)
With this notation in place, we can state a special case of Theorem \ref{thm:minimaltiltingcomplexesquantum}.

\begin{Theorem*}
	Suppose that the Kazhdan-Lusztig functor $F_\ell$ is an equivalence.
	Then for $x,y \in \prescript{S}{}{W_\ell}$ and for a weight $\lambda \in C_\ell \cap X$, we have
	\[ \sum_{i \in \Z} [ C_\mathrm{min}(\Delta^\zeta_{x\Cdot_\ell\lambda})_i : T^\zeta_{y\Cdot_\ell\lambda} ]_\oplus \cdot v^i = n^{x,y} \]
	and
	\[ \sum_{i \in \Z} [ C_\mathrm{min}(L^\zeta_{x\Cdot_\ell\lambda})_i : T^\zeta_{y\Cdot_\ell\lambda} ]_\oplus \cdot v^i = \sum_{z \in W } \overline{m_{z,x}} \cdot n^{z,y} , \]
	where $m_{y,x}$ is the spherical Kazhdan-Lusztig polynomial and $n^{x,y}$ is the inverse anti-spherical Kazhdan-Lusztig polynomial corresponding to $x$ and $y$ with respect to the parabolic subgroup $W$ of $W_\ell$, and where we write $\overline{\phantom{A}} \colon \Z[v^{\pm 1}] \to \Z[v^{\pm 1}]$ for the unique ring homomorphism with $v \mapsto v^{-1}$.
\end{Theorem*}

One further example considered in Section \ref{sec:examples} is the highest weight category $\Rep(G)$ of finite-dimensional rational representations of a simply-connected simple algebraic group over an algebraically closed field of characteristic $p>0$.
Here, the multiplicities of indecomposable tilting objects appearing in the terms of minimal tilting complexes of standard objects and simple objects are generally not governed by Kazhdan-Lusztig polynomials.

We conclude the introduction by giving an outline of the structure of this article.
In Section \ref{sec:preliminaries}, we summarize some preliminary results about highest weight categories.
Our main objects of study, the minimal tilting complexes, will be introduced in Section \ref{sec:minimaltiltingcomplexes}, where we also establish some of their key properties.
In Section \ref{sec:standardsimple}, we develop techniques for computing the multiplicities of indecomposable tilting objects appearing in the terms of minimal tilting complexes of standard objects and simple objects in highest weight categories, and we apply these techniques in Section \ref{sec:examples} to show that the aforementioned multiplicities can be computed in terms of Kazhdan-Lusztig polynomials in some important examples.
More specifically, we do this for blocks of the BGG category $\mathcal{O}$ of a complex semisimple Lie algebra in Subsection \ref{subsec:categoryO}, for blocks of the category $\mathcal{O}$ of an affine Kac-Moody Lie algebra in Subsection \ref{subsec:KacMoody} and for blocks of quantum groups at roots of unity in Subsection \ref{subsec:quantum}.

\subsection*{Acknowledgments}

This article is partly based on my PhD thesis and I would like to express my deepest gratitude to my advisor Donna Testerman for her help and her suggestions throughout the last years, and also for reading a preliminary version of the article.
I would also like to thank Stephen Donkin, Thorge Jensen and Walter Mazorchuk for helpful discussions and comments.
This work was funded by the Swiss National Science Foundation under the grants FNS 200020\_175571 and FNS 200020\_207730 and by the Singapore MOE grant R-146-000-294-133.

\section{Preliminaries} \label{sec:preliminaries}

Let us start by setting up some notation.
For an additive category $\mathcal{A}$, we write $C(\mathcal{A})$ for the category of (cochain) complexes in $\mathcal{A}$ and $K(\mathcal{A})$ for the homotopy category of $\mathcal{A}$.
Similarly, we write $C^b(\mathcal{A})$ for the category of bounded complexes and $K^b(\mathcal{A})$ for the bounded homotopy category.
If $\mathcal{A}$ is abelian then we further denote by $D^b(\mathcal{A})$ the bounded derived category of $\mathcal{A}$.
Given a complex $X = (X_\bullet,d_\bullet)$ in $\mathcal{A}$, we call $X_i$ the term in homological degree $i$ of $X$.
The homological shift of $X$ by $i \in \Z$ is the complex $X[i] = (X^\prime_\bullet,d^\prime_\bullet)$ with terms $X^\prime_j = X_{i+j}$ and differentials $d^\prime_j = (-1)^i \cdot d_{i+j}$ for $j \in \Z$.
Now suppose that $\mathcal{A}$ is abelian.
For objects $A$ and $B$ of $\mathcal{A}$, the $\Ext$-group of degree $i \in \Z$ is
\[ \Ext_\mathcal{C}^i(A,B) \coloneqq \Hom_{D^b(\mathcal{C})}( A , B[i] ) , \]
where we view $A$ and $B$ as complexes concentrated in degree zero, and the $i$-th cohomology object of a complex $X = (X_\bullet,d_\bullet)$ in $\mathcal{A}$ is
\[ H^i(X) = \ker(d_i) / \im(d_{i-1}) . \]
For more background on derived categories and $\Ext$-groups, we refer the reader to \cite{KrauseDerivedCategory}.

\subsection{Highest weight categories}

Highest weight categories were pioneered by E.~Cline, B.J.~Parshall and L.~Scott in \cite{CPSHighestWeight} in an attempt to unify the study of various categories that naturally appear in Lie theoretic representation theory and beyond.
There is by now an abundance of literature on highest weight categories, but not all authors choose to work with the same set of axioms and finiteness assumptions.
The definition that we shall use here is as follows.

Let $\kk$ be a field and let $\mathcal{C}$ be a finite length $\kk$-linear abelian category.
Further suppose that there is a lower finite%
\footnote{A poset $(\Lambda,\leq)$ is called \emph{lower finite} if the set $\{ \mu \in \Lambda \mid \mu \leq \lambda \}$ is finite for all $\lambda \in \Lambda$.
One can also study \emph{upper finite} highest weight categories, as explained in \cite{BrundanStroppel}, but these may fail to have tilting objects.}
poset $(\Lambda,\leq)$ and a map
\[ \Lambda \longrightarrow \mathrm{Ob}(\mathcal{C}) \qquad \qquad  \lambda \longmapsto L_\lambda \] such that the objects $L_\lambda$, $\lambda \in \Lambda$, form a set of representatives for the isomorphism classes of simple objects in $\mathcal{C}$.
For an order ideal $\Gamma \subseteq \Lambda$, the \emph{truncated category} $\mathcal{C}_\Gamma$ is the Serre subcategory of $\mathcal{C}$ generated by the simple objects $\{ L_\lambda \mid \lambda \in \Gamma \}$, and for $\lambda \in \Lambda$, we set
\[ \mathcal{C}_{\leq \lambda} = \mathcal{C}_{ \{ \mu \in \Lambda \mid \mu \leq \lambda \} } \qquad \text{and} \qquad \mathcal{C}_{< \lambda} = \mathcal{C}_{ \{ \mu \in \Lambda \mid \mu < \lambda \} } . \]

\begin{Definition} \label{def:HWC}
We call $\mathcal{C}$ a \emph{highest weight category} with \emph{weight poset} $(\Lambda,\leq)$ if the following conditions are satisfied for all $\lambda,\mu \in \Lambda$:
\begin{description}[labelindent=.2cm]
	\item[(HW1)] The simple object $L_\lambda$ has a projective cover $\Delta_\lambda \to L_\lambda$ and an injective hull $L_\lambda \to \nabla_\lambda$ in $\mathcal{C}_{\leq \lambda}$.
	\item[(HW2)] We have $\Hom_\mathcal{C}(\Delta_\lambda,\nabla_\lambda) \cong \kk$.
	\item[(HW3)] We have $\Ext_\mathcal{C}^1(\Delta_\lambda,\nabla_\mu)=0$ and $\Ext_\mathcal{C}^2(\Delta_\lambda,\nabla_\mu)=0$.
\end{description}
We call $\Delta_\lambda$ the \emph{standard object} and $\nabla_\lambda$ the \emph{costandard object} of highest weight $\lambda \in \Lambda$.
\end{Definition}

Now suppose that $\mathcal{C}$ is a highest weight category with weight poset $(\Lambda,\leq)$.
Observe that the conditions (HW1) and (HW2) imply that
\[ [ \Delta_\lambda : L_\lambda ] = [ \nabla_\lambda : L_\lambda ] = \dim \Hom_\mathcal{C}( \Delta_\lambda , \nabla_\lambda ) = 1 ; \]
in particular the kernel of $\Delta_\lambda \to L_\lambda$ and the cokernel of $L_\lambda \to \nabla_\lambda$ belong to $\mathcal{C}_{<\lambda}$.
Furthermore, as the $\Hom$-space $\Hom_\mathcal{C}(L_\lambda,L_\lambda)$ naturally embeds into $\Hom_\mathcal{C}(\Delta_\lambda,\nabla_\lambda) \cong \kk$, we have
\[ \Hom_\mathcal{C}(L_\lambda,L_\lambda) \cong \kk \]
for all $\lambda \in \Lambda$.
Since all objects of $\mathcal{C}$ have finite length, this implies that all $\Hom$-spaces in $\mathcal{C}$ are finite-dimensional, whence our definition is equivalent to that of a \emph{lower finite highest weight category} in the terminology of \cite[Definition 3.50]{BrundanStroppel} by Corollary 3.64 in \cite{BrundanStroppel}.
This also implies that for any order ideal $\Gamma \subseteq \Lambda$, the truncated category $\mathcal{C}_\Gamma$ is a highest weight category with weight poset $(\Gamma,\leq_\Gamma)$, where $\leq_\Gamma$ denotes the restriction of $\leq$ to $\Gamma$.
Furthermore, by Lemma 3.48 and Theorem 3.59 in \cite{BrundanStroppel}, the condition (HW3) generalizes to
\begin{equation} \label{eq:extvanishing}
	\Ext_\mathcal{C}^i( \Delta_\lambda , \nabla_\mu ) \cong \begin{cases} \kk & \text{if } \lambda = \mu \text{ and } i=0, \\ 0 & \text{otherwise} \end{cases}
\end{equation}
for all $\lambda,\mu \in \Lambda$ and $i \geq 0$.
This $\Ext$-vanishing property is crucial for proving the following lemma:

\begin{Lemma} \label{lem:Extvanishingsimplecostandard}
	Let $\lambda,\mu \in \Lambda$ and $i \geq 0$ such that $\Ext_\mathcal{C}^i( L_\lambda , \nabla_\mu ) \neq 0$.
	Then there exist $\mu_0,\ldots,\mu_i \in \Lambda$ with $\mu = \mu_0 < \mu_1 < \cdots < \mu_i = \lambda$.
\end{Lemma}
\begin{proof}
	This is straightforward to see by induction on $i$, using the $\Ext$-vanishing property \eqref{eq:extvanishing} and the long exact sequence of $\Ext$-groups that one gets by applying the functor $\Hom_\mathcal{C}(-,\nabla_\mu)$ to the short exact sequence $0 \to \rad \, \Delta_\lambda \to \Delta_\lambda \to L_\lambda \to 0$.
	See also Proposition 6.20 in \cite{Jantzen}.
\end{proof}

By the preceding lemma, $\nabla_\lambda$ is an injective hull of $L_\lambda$
in $\mathcal{C}_\Gamma$ for any order ideal $\Gamma \subseteq \Lambda$ such that $\lambda$ is a maximal element of $\Gamma$, because $\Ext_{\mathcal{C}_\Gamma}^1(L_\mu,\nabla_\lambda) = 0$ for all $\mu \in \Gamma$.
Analogously, one sees that $\Delta_\lambda$ is a projective cover of $L_\lambda$
in $\mathcal{C}_\Gamma$ for any order ideal $\Gamma \subseteq \Lambda$ such that $\lambda$ is maximal in $\Gamma$.
We use these observations to prove the following lemma.

\begin{Lemma} \label{lem:standardcostandardcomposition}
	Let $M$ be an object of $\mathcal{C}$ and let $\lambda \in \Lambda$ be maximal among the highest weights of the composition factors of $M$.
	Then we have $\Hom_\mathcal{C}(\Delta_\lambda,M) \neq 0$, and for any non-zero homomorphism $f \colon \Delta_\lambda \to M$, there exists a homomorphism $g \colon M \to \nabla_\lambda$ such that $g \circ f \neq 0$.
\end{Lemma}
\begin{proof}
	By truncating at a suitable order ideal, we may assume that $\lambda$ is maximal in $\Lambda$, so that $\Delta_\lambda$ is a projective cover and $\nabla_\lambda$ is an injective hull of $L_\lambda$.
	Then the first claim is immediate because
	\[ 0 \neq [M:L_\lambda] = \dim \Hom_\mathcal{C}(\Delta_\lambda,M) . \]
	For a non-zero homomorphism $f \colon \Delta_\lambda \to M$, the simple object $L_\lambda$ appears with composition multiplicity one in the image $M^\prime = \im(f)$ of $f$ (because $L_\lambda$ is the unique simple quotient of $\Delta_\lambda$), and the same argument as before shows that $\dim \Hom_\mathcal{C}(M^\prime,\nabla_\lambda) = 1$.
	By injectivity of $\nabla_\lambda$, any non-zero homomorphism $g^\prime \colon M^\prime \to \nabla_\lambda$ lifts to a homomorphism $g \colon M \to \nabla_\lambda$, and the composition $g \circ f$ is non-zero because $g^\prime$ is non-zero and $f$ induces an epimorphism $\Delta_\lambda \to M^\prime$.
\end{proof}

The following result was first observed in \cite[Section 1.5]{CPSabstractKL} (under slightly different assumptions).

\begin{Proposition} \label{prop:derivedtruncated}
	Let $\Gamma \subseteq \Lambda$ be an order ideal and consider the canonical functor
	\[ i_\Gamma \colon \mathcal{C}_\Gamma \longrightarrow \mathcal{C} . \]
	The derived functor
	\[ R \, i_\Gamma \colon D^b(\mathcal{C}_\Gamma) \longrightarrow D^b(\mathcal{C}) \]
	is fully faithful and
	its essential image is the full subcategory $D^b_{\mathcal{C}_\Gamma}(\mathcal{C})$ of $D^b(\mathcal{C})$ whose objects are the complexes in $\mathcal{C}$ with cohomology objects in $\mathcal{C}_\Gamma$.
\end{Proposition}
\begin{proof}
	The derived functor $R \, i_\Gamma$ is fully faithful by Proposition A.7.9 in \cite{RicheHabilitation}, and it is straightforward to show by dévissage (i.e.\ by induction on the number of non-zero cohomology objects) that every complex in $\mathcal{C}$ with cohomology objects in $\mathcal{C}_\Gamma$ belongs to the essential image of $R \, i_\Gamma$.
\end{proof}

\subsection{Filtrations and tilting objects}

Next we discuss the concept of \emph{$\Delta$-filtrations} and \emph{$\nabla$-filtrations} (sometimes called \emph{Weyl filtrations} and \emph{good filtrations}) in highest weight categories.

\begin{Definition}
	A \emph{$\Delta$-filtration} of an object $M$ of $\mathcal{C}$ is a filtration
	\[ 0 = M_0 \subseteq M_1 \subseteq \cdots \subseteq M_r = M \]
	such that $M_i / M_{i-1} \cong \Delta_{\lambda_i}$ for some $\lambda_i \in \Lambda$, for $i=1,\ldots,r$.
	A \emph{$\nabla$-filtration} of an object $N$ of $\mathcal{C}$ is a filtration
	\[ 0 = N_0 \subseteq N_1 \subseteq \cdots \subseteq N_s = N \]
	such that $N_i / N_{i-1} \cong \nabla_{\mu_i}$ for some $\mu_i \in \Lambda$, for $i=1,\ldots,s$.
	We write $\mathcal{C}_\Delta$ and $\mathcal{C}_\nabla$ for the full subcategories of objects of $\mathcal{C}$ that admit a $\Delta$-filtration or a $\nabla$-filtration, respectively.
\end{Definition}

The objects of $\mathcal{C}$ that admit a $\Delta$-filtration can be characterized by certain $\Ext$-vanishing properties.
This result is known as \emph{Donkin's cohomological criterion} and was first proven (for rational representations of reductive algebraic groups) by S.\ Donkin in \cite{DonkinFiltrationRationalModules}.
A proof in our setting can be found in \cite[Proposition A.7.10]{RicheHabilitation}.

\begin{Proposition} \label{prop:Donkincohomological}
	For an object $M$ of $\mathcal{C}$, the following are equivalent:
	\begin{enumerate}
		\item $M$ admits a $\Delta$-filtration;
		\item $\Ext_{\mathcal{C}}^1(M,\nabla_\lambda)=0$ for all $\lambda \in \Lambda$;
		\item $\Ext_{\mathcal{C}}^i(M,\nabla_\lambda)=0$ for all $\lambda \in \Lambda$ and $i>0$.
	\end{enumerate}
\end{Proposition}

There is an analogous cohomological criterion for the existence of $\nabla$-filtrations which we leave to the reader.
Observe that Proposition \ref{prop:Donkincohomological} implies that all projective objects in $\mathcal{C}$ admit $\Delta$-filtrations, and analogously, all injective objects in $\mathcal{C}$ admit $\nabla$-filtrations.
Furthermore, one easily sees that for $\lambda \in \Lambda$ and an object $M$ of $\mathcal{C}$ that admits a Weyl filtration, the multiplicity of $\Delta_\lambda$ in any Weyl filtration
\[ 0 = M_0 \subseteq \cdots \subseteq M_r = M \]
of $M$, where $M_i / M_{i-1} \cong \Delta_{\lambda_i}$ for $i = 1 , \ldots , r$, is given by
\[ [ M : \Delta_\lambda ]_\Delta \coloneqq \dim \Hom_\mathcal{C}( M , \nabla_\lambda ) = \{ i \mid \lambda_i = \lambda \} . \]
As before, there is an analogous result for $\nabla$-filtrations which we leave to the reader.

\begin{Definition}
	An object $M$ of $\mathcal{C}$ is called a \emph{tilting object} if it admits both a $\Delta$-filtration and a $\nabla$-filtration.
	We write $\Tilt(\mathcal{C})$ for the full subcategory of tilting objects in $\mathcal{C}$.
\end{Definition}

Using Proposition \ref{prop:Donkincohomological} (and the analogous statement for $\nabla$-filtrations), it is straightforward to see that the direct sum $M \oplus N$ of two objects $M$ and $N$ of $\mathcal{C}$ is a tilting object if and only if $M$ and $N$ are both tilting objects.
In particular, every tilting object can be written as a finite direct sum of indecomposable tilting objects.
The following classification of indecomposable tilting objects is due to C.M.\ Ringel \cite{RingelAlmostSplit}; in our setting, it is proven in Theorem A.7.14 in \cite{RicheHabilitation}.

\begin{Theorem}[Ringel]
	For all $\lambda \in \Lambda$, there exists a unique (up to isomorphism) indecomposable tilting object $T_\lambda$ in $\mathcal{C}$ such that $T_\lambda$ belongs to $\mathcal{C}_{\leq \lambda}$ and $[T_\lambda:L_\lambda] = 1$.
	Furthermore, every indecomposable tilting object in $\mathcal{C}$ is isomorphic to $T_\mu$ for a unique $\mu \in \Lambda$. 
\end{Theorem}

The following observation will be the cornerstone of our study of minimal tilting complexes below; see \cite[Proposition A.7.17]{RicheHabilitation} for a proof.

\begin{Proposition} \label{prop:tiltinghomotopyderived}
	The canonical functor
	\[ \mathfrak{T} \colon K^b\big( \Tilt(\mathcal{C}) \big) \longrightarrow D^b( \mathcal{C} ) \]
	from the bounded homotopy category of $\Tilt(\mathcal{C})$ to the bounded derived category of $\mathcal{C}$ is an equivalence of triangulated categories.
\end{Proposition}

\subsection{Ringel duality}

When studying tilting objects in a highest weight category with finite weight poset, it is often useful to use the \emph{Ringel dual}.
The material presented below is originally due to C.M.\ Ringel, see Section 6 in \cite{RingelAlmostSplit}.
We follow the more general exposition of Section 4 in \cite{BrundanStroppel}.

For this subsection, suppose that $\Lambda$ is finite, so that the simple object $L_\lambda$ has a projective cover $P_\lambda \to L_\lambda$ and an injective hull $L_\lambda \to I_\lambda$, for all $\lambda \in \Lambda$ (see Theorem 3.2.1 in \cite{BeilinsonGinzburgSoergel}).
Consider the \emph{characteristic tilting module} $T \coloneqq \bigoplus_{\lambda \in \Lambda} T_\lambda$ and the algebra $A \coloneqq \End_\mathcal{C}(T)^\mathrm{op}$.

\begin{Definition}
	The \emph{Ringel dual} of $\mathcal{C}$ is the category $\mathcal{C}^\prime$ of finite-dimensional left $A$-modules.
\end{Definition}

There are two canonical \emph{Ringel duality functors} from $\mathcal{C}$ to $\mathcal{C}^\prime$, given by
\[ R_\mathcal{C}^\Delta = \Hom_\mathcal{C}(T,-) \qquad \text{and} \qquad R_\mathcal{C}^\nabla = \Hom_\mathcal{C}(-,T)^* . \]
Observe that we have to take the dual in the definition of $R_\mathcal{C}^\nabla$ because $\Hom_\mathcal{C}(-,T)$ defines a functor from $\mathcal{C}$ to the category of right $A$-modules.
Some important properties of Ringel duality are summarized in the following result; see Theorem 4.10 in \cite{BrundanStroppel} and its proof.

\begin{Theorem} \label{thm:Ringeldual}
	The Ringel dual $\mathcal{C}^\prime$ of $\mathcal{C}$ is a highest weight category with weight poset $(\Lambda,\leq^\mathrm{op})$.
	The standard object and costandard object of highest weight $\lambda \in \Lambda$ in $\mathcal{C}^\prime$ are given by
	\[ \Delta^\prime_\lambda \coloneqq R_\mathcal{C}^\Delta( \nabla_\lambda ) \qquad \text{and} \qquad \nabla^\prime_\lambda \coloneqq R_\mathcal{C}^\nabla( \Delta_\lambda ) \]
	and the corresponding simple object is $L^\prime_\lambda \cong \head \, \Delta^\prime_\lambda \cong \soc \, \nabla^\prime_\lambda$.
	Furthermore, the Ringel duality functors restrict to equivalences
	\[ R_\mathcal{C}^\Delta \colon \mathcal{C}_\nabla \longrightarrow \mathcal{C}^\prime_\Delta \qquad \text{and} \qquad R_\mathcal{C}^\nabla \colon \mathcal{C}_\Delta \longrightarrow \mathcal{C}^\prime_\nabla \]
	which take acyclic complexes in $\mathcal{C}_\nabla$ (respectively $\mathcal{C}_\Delta$) to acyclic complexes in $\mathcal{C}^\prime_\Delta$ (respectively $\mathcal{C}^\prime_\nabla$), and we have
	\[ R_\mathcal{C}^\nabla( T_\lambda ) = I^\prime_\lambda , \qquad R_\mathcal{C}^\Delta( T_\lambda ) = P^\prime_\lambda , \qquad R_\mathcal{C}^\nabla( P_\lambda ) \cong R_\mathcal{C}^\Delta( I_\lambda ) \cong T^\prime_\lambda \]
	for all $\lambda \in \Lambda$, where we write $I^\prime_\lambda$ and $P^\prime_\lambda$ for the injective hull and the projective cover of $L^\prime_\lambda$ in $\mathcal{C}^\prime$, respectively, and $T^\prime_\lambda$ for the indecomposable tilting object of highest weight $\lambda$ in $\mathcal{C}^\prime$.
\end{Theorem}

\subsection{\texorpdfstring{$\nabla$}{Costandard}-filtration dimension and \texorpdfstring{$\Delta$}{standard}-filtration dimension}

We conclude this section with a brief discussion of the notion of $\nabla$-filtration dimension and $\Delta$-filtration dimension, which was introduced by E.\ Friedlander and B.J.\ Parshall in \cite{FriedlanderParshallGFD} in order to study the cohomology of Lie algebras and algebraic groups in positive characteristic.

\begin{Definition}
	Let $M$ be a non-zero object of $\mathcal{C}$.
	The \emph{$\nabla$-filtration dimension} of $M$ is
	\[ \dim_\nabla(M) = \max \{ d \mid \Ext_\mathcal{C}^d( \Delta_\lambda , M ) \neq 0 \text{ for some } \lambda \in \Lambda \} \]
	and the \emph{$\Delta$-filtration dimension} of $M$ is
	\[ \dim_\Delta(M) = \max \{ d \mid \Ext_\mathcal{C}^d(M,\nabla_\lambda) \neq 0 \text{ for some } \lambda \in \Lambda \} . \]
\end{Definition}

The $\nabla$-filtration dimension and the $\Delta$-filtration dimension of a non-zero object $M$ of $\mathcal{C}$ are well-defined because the sets
\[ \{ d \mid \Ext_\mathcal{C}^d( \Delta_\lambda , M ) \neq 0 \text{ for some } \lambda \in \Lambda \} \qquad \text{and} \qquad \{ d \mid \Ext_\mathcal{C}^d( M , \nabla_\lambda ) \neq 0 \text{ for some } \lambda \in \Lambda \} \]
are bounded and non-empty by Lemmas \ref{lem:Extvanishingsimplecostandard} and \ref{lem:standardcostandardcomposition}.
We alert the reader to the fact that $\nabla$-filtrations are often called \emph{good filtrations} in the literature, and $\Delta$-filtrations are often called \emph{Weyl filtrations}.
Accordingly, the $\nabla$-filtration dimension and the $\Delta$-filtration dimension are sometimes called \emph{good filtration dimension} and \emph{Weyl filtration dimension} and denoted by $\mathrm{gfd}(M)$ and $\mathrm{wfd}(M)$, respectively.

\begin{Remark} \label{rem:gfdzero}
	Observe that by Proposition \ref{prop:Donkincohomological}, a non-zero object $M$ of $\mathcal{C}$ admits a $\Delta$-filtration if and only if $\dim_\Delta( M ) = 0$, and $M$ admits a $\nabla$-filtration if and only if $\dim_\nabla( M ) = 0$.
\end{Remark}

\begin{Lemma} \label{lem:GFDSES}
	For any short exact sequence $0 \to A \to B \to C \to 0$ of non-zero objects in $\mathcal{C}$, we have
	\begin{enumerate}
		\item $\dim_\nabla(A) \leq \max\{ \dim_\nabla(B) , \dim_\nabla(C)+1 \}$ with equality if $\dim_\nabla(B) \neq \dim_\nabla(C)$;
		\item $\dim_\nabla(B) \leq \max\{ \dim_\nabla(A) , \dim_\nabla(C) \}$ with equality if $\dim_\nabla(A) \neq \dim_\nabla(C)+1$;
		\item $\dim_\nabla(C) \leq \max\{ \dim_\nabla(A)-1 , \dim_\nabla(B) \}$ with equality if $\dim_\nabla(A) \neq \dim_\nabla(B)$;
		\item $\dim_\Delta(A) \leq \max\{ \dim_\Delta(B) , \dim_\Delta(C)-1 \}$ with equality if $\dim_\Delta(B) \neq \dim_\Delta(C)$;
		\item $\dim_\Delta(B) \leq \max\{ \dim_\Delta(A) , \dim_\Delta(C) \}$ with equality if $\dim_\Delta(C) \neq \dim_\Delta(A)+1$;
		\item $\dim_\Delta(C) \leq \max\{ \dim_\Delta(B) , \dim_\Delta(A)+1 \}$ with equality if $\dim_\Delta(A) \neq \dim_\Delta(B)$.
	\end{enumerate}
\end{Lemma}
\begin{proof}
	We only prove (1), the proofs of (2)--(6) are completely analogous.
	For $i\geq 0$ and $\lambda \in \Lambda$, the short exact sequence
	$ 0\to A\to B \to C \to 0 $
	gives rise to an exact sequence
	\[ \Ext_\mathcal{C}^i\big( \Delta_\lambda , B \big) \to \Ext_\mathcal{C}^i\big( \Delta_\lambda , C \big) \to \Ext_\mathcal{C}^{i+1}\big( \Delta_\lambda , A \big) \to \Ext_\mathcal{C}^{i+1}\big( \Delta_\lambda , B \big) \to \Ext_\mathcal{C}^{i+1}\big( \Delta_\lambda , C \big) . \]
	If $i+1 > \max\{ \dim_\nabla(B) , \dim_\nabla(C)+1 \}$ then
	$ \Ext_\mathcal{C}^i( \Delta_\lambda , C ) = 0 $ and $ \Ext_\mathcal{C}^{i+1}( \Delta_\lambda , B ) = 0 $
	for all $\lambda \in \Lambda$, and we conclude that $\Ext_\mathcal{C}^{i+1}( \Delta_\lambda , A ) = 0$ for all $\mu\in \Lambda$, hence
	\[ \dim_\nabla(A) \leq \max\{ \dim_\nabla(B) , \dim_\nabla(C)+1 \} . \]
	If $\dim_\nabla(B)<\dim_\nabla(C) \eqqcolon d$ then $\Ext_\mathcal{C}^d( \Delta_\lambda , C ) \neq 0$ for some $\lambda \in X^+$, and as $\Ext_\mathcal{C}^d( \Delta_\lambda , B ) = 0$, the $\Ext$-group $\Ext_\mathcal{C}^d( \Delta_\lambda , C )$ embeds into $\Ext_\mathcal{C}^{d+1}( \Delta_\lambda , A )$.
	This implies that $\Ext_\mathcal{C}^{d+1}( \Delta_\lambda , A ) \neq 0$ and therefore $\dim_\nabla(A)=d+1$.
	Analogously, if $\dim_\nabla(C)<\dim_\nabla(B)\eqqcolon d^\prime$ then $\Ext_\mathcal{C}^{d^\prime}( \Delta_\lambda , B ) \neq 0$ for some $\lambda \in \Lambda$, and $\Ext_\mathcal{C}^{d^\prime}( \Delta_\lambda , A )$ surjects onto $\Ext_\mathcal{C}^{d^\prime}( \Delta_\lambda , B )$ because $\Ext_\mathcal{C}^{d^\prime}( \Delta_\lambda , C ) = 0$.
	As before, we conclude that $\Ext_\mathcal{C}^{d^\prime}( \Delta_\lambda , A ) \neq 0$ and $\dim_\nabla(A)=d^\prime$.
\end{proof}

\section{Minimal tilting complexes} \label{sec:minimaltiltingcomplexes}

In this section, we explain the concept of \emph{minimal tilting complexes} and prove some of their elementary properties.
The main idea is to transport information from $\mathrm{Tilt}( \mathcal{C} )$ to $\mathcal{C}$ using the equivalence
\[ \mathfrak{T} \colon K^b\big( \mathrm{Tilt}(\mathcal{C}) \big) \longrightarrow D^b( \mathcal{C} ) \]
from Proposition \ref{prop:tiltinghomotopyderived}.
In order to do this it will be helpful to choose a unique representative from every homotopy class in $K^b\big( \mathrm{Tilt}( \mathcal{C} ) \big)$ that is minimal in a suitable sense.

\subsection{Minimal complexes in additive categories} \label{subsec:minimalcomplexes}

The theory of minimal complexes is widely used in homological algebra, but we were not able to find a reference where it is discussed in the greatest possible generality.
The approach that we present here shows that minimal complexes are well behaved (if they exist) in any additive category, and we use ideas of D.\ Bar-Natan from \cite{BarNatan} to show that minimal complexes always exist in a Krull-Schmidt category.
Let $\mathcal{A}$ be an additive category.

\begin{Definition}
	The \emph{radical} of $\mathcal{A}$ is the ideal%
	\footnote{An \emph{ideal} $\mathcal{I}$ in an additive category $\mathcal{A}$ is a collection of subgroups $\mathcal{I}(A,B) \subseteq \Hom_\mathcal{A}(A,B)$, for all $A, B \in \mathrm{Ob}(\mathcal{A})$, such that for all $f \in \mathcal{I}(A,B)$ and for all homomorphisms $a \colon A^\prime \to A$ and $b \colon B \to B^\prime$ in $\mathcal{A}$, we have $b \circ f \circ a \in \mathcal{I}(A^\prime,B^\prime)$.}
	$\rad_\mathcal{A}$ with
	\[ \mathrm{rad}_\mathcal{A}(A,B) = \big\{ f\in\Hom_\mathcal{A}(A,B) \mathop{\big|} b \circ f \circ a \in J\big(\End_\mathcal{A}(C)\big) \text{ for all } a \colon C \to A \text{ and } b \colon B \to C \big\} \]
	for all objects $A$ and $B$ of $\mathcal{A}$, where $J\big(\End_\mathcal{A}(C)\big)$ denotes the Jacobson radical of the endomorphism ring $\End_\mathcal{A}(C)$ of an object $C$ of $\mathcal{A}$.
\end{Definition}

For a general discussion of the radical, see Section 2 in \cite{KrauseKrullSchmidt}.

\begin{Definition}
	A complex
	\[ \cdots \xrightarrow{\,d_{-2}\,} A_{-1} \xrightarrow{\,d_{-1}\,} A_0 \xrightarrow{~d_0~} A_1 \xrightarrow{~d_1~} \cdots \]
	in $\mathcal{A}$ is called \emph{minimal} if $d_i\in \mathrm{rad}_\mathcal{A}(A_i,A_{i+1})$ for all $i\in\Z$.
\end{Definition}

We first prove two results that show that every homotopy class in $K(\mathcal{A})$ contains at most one minimal complex.

\begin{Lemma} \label{lem:minimalcomplexsplit}
	Let $C$ and $C^\prime$ be complexes in $\mathcal{A}$ and let $f\colon C \to C^\prime$ be a morphism of complexes.
	\begin{enumerate}
		\item If $C$ is a minimal complex and $f$ is a split monomorphism in $K(\mathcal{A})$ then $f$ is also a split monomorphism in $C(\mathcal{A})$.
		\item If $C^\prime$ is a minimal complex and $f$ is a split epimorphism in $K(\mathcal{A})$ then $f$ is also a split epimorphism in $C(\mathcal{A})$.
	\end{enumerate}
\end{Lemma}
\begin{proof}
	Let us write $C=(A_\bullet,d_\bullet)$ and suppose that $d_i\in\mathrm{rad}_\mathcal{A}(A_i,A_{i+1})$ for all $i \in \Z$.
	If $f$ is a split monomorphism in the homotopy category $K(\mathcal{A})$ then there exist a morphism of complexes $g \colon C^\prime \to C$ and a homotopy equivalence $h=(h_i)_{i\in\Z}$ from $g \circ f$ to $\id_C$, so
	\[ \id_{A_i} - g_i \circ f_i = h_{i+1} \circ d_i + d_{i-1} \circ h_i \]
	for all $i\in\Z$. Now $h_{i+1} \circ d_i + d_{i-1} \circ h_i \in \rad_\mathcal{A}(A_i,A_i) \subseteq J\big( \End_\mathcal{A}(A_i) \big)$, and it follows that $\varphi_i \coloneqq g_i \circ f_i$ is invertible.
	Then $g^\prime \coloneqq (\varphi_i^{-1} \circ g_i)_{i\in\Z}$ is a morphism of complexes and $g^\prime \circ f = \id_C$, so $f$ is a split monomorphism in $C(\mathcal{A})$.
	The second claim can be proven analogously.
\end{proof}

\begin{Corollary} \label{cor:minimalcomplexunique}
	Let $C$ and $C^\prime$ be minimal complexes over $\mathcal{A}$ and let $f\colon C \to C^\prime$ be a morphism of complexes.
	If $f$ is an isomorphism in $K(\mathcal{A})$ then $f$ is an isomorphism in $C(\mathcal{A})$.  
\end{Corollary}
\begin{proof}
	By Lemma \ref{lem:minimalcomplexsplit}, $f$ has a left inverse and a right inverse in $C(\mathcal{A})$. It is straightforward to check that these must coincide, so $f$ is invertible, as claimed.
\end{proof}

From now on, suppose that $\mathcal{A}$ is a Krull-Schmidt category, i.e.\ that every object of $\mathcal{A}$ can be written as a finite direct sum of objects with local endomorphism rings.
Then every object of $\mathcal{A}$ admits a \emph{Krull-Schmidt decomposition} as a finite direct sum of indecomposable objects, and the multiplicity of a given indecomposable object in such a decomposition is independent of the chosen decomposition; see Theorem 4.2 in \cite{KrauseKrullSchmidt}.
We denote the multiplicity of an indecomposable object $N$ of $\mathcal{A}$ in a Krull-Schmidt decomposition of an object $M$ of $\mathcal{A}$ by $[M : N]_\oplus$.
In this setting we can give an alternative characterization of the radical of $\mathcal{A}$ as follows:

\begin{Lemma} \label{lem:radical}
	Let $A$ and $B$ be objects of $\mathcal{A}$ and $f\in\Hom_\mathcal{A}(A,B)$. The following are equivalent:
	\begin{enumerate}
		\item $f\in \mathrm{rad}_\mathcal{A}(A,B)$;
		\item no isomorphism between non-zero objects of $\mathcal{A}$ factors through $f$;
		\item no isomorphism between indecomposable objects of $\mathcal{A}$ factors through $f$.
	\end{enumerate}
\end{Lemma}
\begin{proof}
	If some isomorphism $g\colon C \to D$ between non-zero objects of $\mathcal{A}$ factors through $f$ then so does $\id_C=g^{-1} \circ g$, so we can write $\id_C=b\circ f \circ a$ for certain morphisms $a\colon C \to A$ and $b \colon B \to C$.
	As $\id_C$ does not belong to the Jacobson radical of $\End_\mathcal{A}(C)$, we conclude that $f \notin \rad_\mathcal{A}(A,B)$ and that (1) implies (2).
	It is obvious that (2) implies (3).
	Now assume (3) and fix Krull-Schmidt decompositions
	\[ A = A_1 \oplus \cdots \oplus A_m \qquad \text{and} \qquad B = B_1 \oplus \cdots \oplus B_n . \]
	Then we can write $f = (f_{ij})_{i,j}$ with $f_{ij} \in \Hom_\mathcal{A}(A_j,B_i)$, and $f_{ij}$ is non-invertible because $f_{ij}$ factors through $f$ for all $1 \leq i \leq n$ and $1 \leq j \leq m$.
	It follows that $f \in \rad(A,B)$ because
	\[ \rad_\mathcal{A}(A,B) = \bigoplus_{i,j} \rad_\mathcal{A}(A_j,B_i) \]
	and $\rad_\mathcal{A}(A_j,B_i)$ is the set of non-invertible homomorphism from $A_j$ to $B_i$, as observed in the discussion after Corollary 4.4 in \cite{KrauseKrullSchmidt}.
\end{proof}

Next we show that every bounded complex over a Krull-Schmidt category is homotopy equivalent to a unique minimal complex.
The key tool in the proof is D.\ Bar-Natan's `Gaussian elimination on complexes'; see \cite[Lemma 4.2]{BarNatan}.

\begin{Lemma} \label{lem:minimalcomplex}
	Every bounded complex $C$ over $\mathcal{A}$ is homotopy equivalent to a minimal complex $C_\mathrm{min}$, and $C_\mathrm{min}$ is unique up to isomorphism in the category of complexes $C^b(\mathcal{A})$.
	Furthermore, $C_\mathrm{min}$ is a direct summand of $C$ in $C^b(\mathcal{A})$.
\end{Lemma}
\begin{proof}
	The uniqueness statement is clear from Corollary \ref{cor:minimalcomplexunique}.
	If $C$ is minimal then there is nothing to show, so now write $C=(A_\bullet,d_\bullet)$ and suppose that $d_i \notin \mathrm{rad}_\mathcal{A}(A_i,A_{i+1})$ for some $i\in\Z$.
	By Lemma \ref{lem:radical}, there exists an indecomposable object $M$ of $\mathcal{A}$ such that $\id_M$ factors through $d_i$, so $\id_M =b \circ d_i \circ a$ for some $a\colon M \to A_i$ and $b\colon A_{i+1} \to M$.
	Then $a$ is a split monomorphism and $b$ is a split epimorphism, so
	\[ A_i \cong B_i \oplus M \qquad \text{and} \qquad A_{i+1} \cong B_{i+1} \oplus M \]
	for certain objects $B_i$ and $B_{i+1}$ of $\mathcal{A}$.
	Consequently, we can write $C$ as
	\begin{center}
		\begin{tikzpicture}[scale=.8]
		\node (Z1) at (-2,0) {$\cdots$};
		\node (A1) at (0,0) {$A_{i-1}$};
		\node (B1) at (4,0) {$B_i \oplus M$};
		\node (C1) at (9,0) {$B_{i+1} \oplus M$};
		\node (D1) at (13,0) {$A_{i+2}$};
		\node (E1) at (15,0) {$\cdots$,};
		\draw[->] (Z1) -- (A1);
		\draw[->] (A1) -- node[above]{\footnotesize $\begin{pmatrix}
			f_1 \\ f_2
			\end{pmatrix}$} (B1);
		\draw[->] (B1) -- node[above]{\footnotesize $\begin{pmatrix}
			d_{11} & d_{12} \\ d_{21} & \id_M
			\end{pmatrix}$} (C1);
		\draw[->] (C1) -- node[above]{\footnotesize $\begin{pmatrix}
			g_1 & g_2
			\end{pmatrix}$} (D1);
		\draw[->] (D1) -- (E1);
		\end{tikzpicture}
	\end{center}
	and using `Gaussian elimination on complexes' as in \cite[Lemma 4.2]{BarNatan},
	we see that $C$ is isomorphic to a complex of the form
	\begin{equation} \label{eq:directsumofcomplexes}
		\begin{tikzpicture}[baseline={(current bounding box.center)},scale=.8]
			\node (Z1) at (-2,0) {$\cdots$};
			\node (A1) at (0,0) {$A_{i-1}$};
			\node (B1) at (4,0) {$B_i \oplus M$};
			\node (C1) at (9,0) {$B_{i+1} \oplus M$};
			\node (D1) at (13,0) {$A_{i+2}$};
			\node (E1) at (15,0) {$\cdots$,};
			\draw[->] (Z1) -- (A1);
			\draw[->] (A1) -- node[above]{\footnotesize $\begin{pmatrix}
				f_1 \\ 0
				\end{pmatrix}$} (B1);
			\draw[->] (B1) -- node[above]{\footnotesize $\begin{pmatrix}
				\psi & 0 \\ 0 & \id_M
				\end{pmatrix}$} (C1);
			\draw[->] (C1) -- node[above]{\footnotesize $\begin{pmatrix}
				g_1 & 0
				\end{pmatrix}$} (D1);
			\draw[->] (D1) -- (E1);
		\end{tikzpicture}
	\end{equation}
	where $\psi=d_{11} - d_{12} \circ d_{21}$.
	Now the complex in \eqref{eq:directsumofcomplexes} is isomorphic to the direct sum of the complexes
	\[ C^\prime \colon ~ \cdots \to A_{i-1} \xrightarrow{\,f_1\,} B_i \xrightarrow{\,\psi\,} B_{i+1} \xrightarrow{\,g_1\,} A_{i+2} \to \cdots \qquad \text{and} \qquad C^{\prime\prime} \colon ~ \cdots \to 0 \to M \xrightarrow{\,\id_M\,} M \to 0 \to \cdots \]
	(with $M$ in homological degrees $i$ and $i+1$). As $C^{\prime\prime}$ is homotopy equivalent to the zero complex, we conclude that $C$ is homotopy equivalent to $C^\prime$.
	Now the existence of $C_\mathrm{min}$ easily follows by induction on the sum of the numbers of indecomposable direct summands of the terms in $C$.
	The final claim is a consequence of the construction, since $C^\prime$ is a direct summand of $C$ in $C^b(\mathcal{A})$.
	Alternatively, we can just note that a homotopy equivalence between $C_\mathrm{min}$ and $C$ is a split monomorphism in $K^b(\mathcal{A})$, and thus also a split monomorphism in $C^b(\mathcal{A})$ by Lemma \ref{lem:minimalcomplexsplit}.
\end{proof}

\begin{Definition}
	Let $C$ be a complex over $\mathcal{A}$ and let $C_\mathrm{min}$ be the unique minimal complex in the homotopy class of $C$. We say that $C_\mathrm{min}$ is the \emph{minimal complex of $C$}.
\end{Definition}

\begin{Corollary} \label{cor:minimalcomplexdirectsummandbound}
	Let $C$ be a bounded complex over $\mathcal{A}$, with minimal complex $C_\mathrm{min}$, and let $M$ be an indecomposable object of $\mathcal{A}$.
	Write $C$ and $C_\mathrm{min}$ as
	\[ \cdots \longrightarrow A_i \longrightarrow A_{i+1} \longrightarrow \cdots \qquad \text{and} \qquad \cdots \longrightarrow B_i \longrightarrow B_{i+1} \longrightarrow \cdots , \]
	respectively.
	Then
	\[ [ A_i : M ]_\oplus \geq [ B_i : M ]_\oplus \geq [ A_i : M ]_\oplus - [ A_{i-1} : M ]_\oplus - [ A_{i+1} : M ]_\oplus \]
	for all $i\in\Z$.
\end{Corollary}
\begin{proof}
	By the uniqueness of minimal complexes and by the proof of Lemma \ref{lem:minimalcomplex}, we can obtain $C_\mathrm{min}$ from $C$ by successively removing pairs of isomorphic indecomposable direct summands from two adjacent terms $A_j$ and $A_{j+1}$ of $C$.
	The maximum number of times that an indecomposable direct summand isomorphic to $M$ can be removed from $A_i$ in this fashion is $[ A_{i-1} : M ]_\oplus + [ A_{i+1} : M ]_\oplus$, and the claim follows.
\end{proof}

\subsection{Minimal complexes of tilting objects} \label{subsec:minimalcomplexestilting}

Now suppose that $\mathcal{C}$ is a highest weight category with weight poset $(\Lambda,\leq)$.
As $\mathrm{Tilt}(\mathcal{C})$ is a Krull-Schmidt category, the theory of minimal complexes explained in the previous subsection can be applied.
Let $X$ be a bounded complex in $\mathcal{C}$.
Then $X$ corresponds to a unique homotopy class in $K^b\big( \Tilt(\mathcal{C}) \big)$ under the equivalence
\[ \mathfrak{T} \colon K^b\big( \Tilt(\mathcal{C}) \big) \longrightarrow D^b( \mathcal{C} ) \]
from Proposition \ref{prop:tiltinghomotopyderived}, and by Lemma \ref{lem:minimalcomplex}, this homotopy class contains a unique minimal complex, up to isomorphism in $C^b( \mathcal{C} )$.

\begin{Definition}
	For a bounded complex $X$ in $\mathcal{C}$, we write $C_\mathrm{min}(X)$ for the unique (up to isomorphism) minimal bounded complex of tilting objects with $C_\mathrm{min}(X) \cong X$ in $D^b(\mathcal{C})$, and we call $C_\mathrm{min}(X)$ the \emph{minimal tilting complex of $X$}.
	For an object $M$ of $\mathcal{C}$, we define the \emph{minimal tilting complex $C_\mathrm{min}(M)$ of $M$} by viewing $M$ as a complex concentrated in degree zero. 
\end{Definition}

\begin{Remark} \label{rem:isomorphicinderivedcohomology}
	For an object $M$ of $\mathcal{C}$ and a bounded complex $X$ in $\mathcal{C}$, we have $X \cong M$ in $D^b( \mathcal{C} )$ if and only if $H^0(X) \cong M$ and $H^i(X) = 0$ for $i \neq 0$.
	In particular, the minimal tilting complex $C_\mathrm{min}(M)$ of $M$ is the unique minimal bounded complex of tilting objects with
\[ H^i\big( C_\mathrm{min}(M) \big) \cong \begin{cases}
M & \text{if } i=0 , \\ 0 & \text{otherwise} .
\end{cases} \]
\end{Remark}

\begin{Remark} \label{rem:minimaltiltingcomplexoftilting}
	For any tilting object $M$ of $\mathcal{C}$, we have $C_\mathrm{min}(M) = M$, where we view $M$ as a complex concentrated in degree zero.
\end{Remark}

\subsection{General properties}

In the following, we will establish some important properties of minimal tilting complexes.
We start by observing that these complexes are well-behaved with respect to direct sums.

\begin{Lemma} \label{lem:minimaltiltingcomplex}
	Let $X$ and $Y$ be bounded complexes in $\mathcal{C}$.
	\begin{enumerate}
		\item We have $C_\mathrm{min}( X \oplus Y )\cong C_\mathrm{min}( X ) \oplus C_\mathrm{min}( Y )$ in $C^b\big( \mathrm{Tilt}( \mathcal{C} ) \big)$.
		\item If $C$ is a bounded complex of tilting objects with $C \cong X$ in $D^b( \mathcal{C} )$ then $C_\mathrm{min}(X)$ is the minimal complex of $C$ and there is a split monomorphism $C_\mathrm{min}(X) \to C$ in $C^b\big( \mathrm{Tilt}(\mathcal{C}) \big)$.
	\end{enumerate}
\end{Lemma}
\begin{proof}
	The first claim follows from the observation that a direct sum of minimal complexes is minimal.
	If $C$ is a bounded complex of tilting objects with $C \cong X$ in $D^b( \mathcal{C} )$ then we also have $C \cong C_\mathrm{min}(X)$  in $D^b( \mathcal{C} )$.
	Using the equivalence from Proposition \ref{prop:tiltinghomotopyderived}, it follows that $C \cong C_\mathrm{min}(X)$ in $K^b\big( \mathrm{Tilt}(\mathcal{C}) \big)$, whence $C_\mathrm{min}(X)$ is the minimal complex of $C$.
	By Lemma \ref{lem:minimalcomplexsplit}, any homotopy equivalence from $C_\mathrm{min}(X)$ to $C$ is a split monomorphism in $C^b\big( \Tilt(\mathcal{C}) \big)$.
\end{proof}

Recall that for an order ideal $\Gamma \subseteq \Lambda$, the truncated subcategory $\mathcal{C}_\Gamma$ of $\mathcal{C}$ is also a highest weight category.
The next lemma shows that the minimal tilting complex $C_\mathrm{min}(M)$ of an object $M$ of $\mathcal{C}$ determines whether $M$ belongs to $\mathcal{C}_\Gamma$.

\begin{Lemma} \label{lem:minimaltiltingcomplexorderideal}
	Let $\Gamma \subseteq \Lambda$ be an order ideal.
	An object $M$ of $\mathcal{C}$ belongs to $\mathcal{C}_\Gamma$ if and only if all terms of the minimal tilting complex $C_\mathrm{min}(M)$ belong to $\mathcal{C}_\Gamma$.
\end{Lemma}
\begin{proof}
	This is straightforward to see from the definition of minimal tilting complexes and the commutative diagram of functors
	\begin{center}
		\begin{tikzpicture}
		\node (C1) at (0,0) {$K^b\big( \Tilt( \mathcal{C}_\Gamma ) \big)$};
		\node (C2) at (4,0) {$D^b( \mathcal{C}_\Gamma )$};
		\node (C3) at (0,-2) {$K^b\big( \Tilt( \mathcal{C} ) \big)$};
		\node (C4) at (4,-2) {$D^b( \mathcal{C} )$};

		\draw[->] (C1) -- (C2);
		\draw[->] (C1) -- (C3);
		\draw[->] (C2) -- (C4);
		\draw[->] (C3) -- (C4);
		\end{tikzpicture}
	\end{center}
	where the vertical arrows are induced by the canonical functor $\mathcal{C}_\Gamma \to \mathcal{C}$ and the horizontal arrows are given by the equivalence from Proposition \ref{prop:tiltinghomotopyderived} for $\mathcal{C}_\Gamma$ and $\mathcal{C}$, respectively.
\end{proof}

Next we observe that the $\nabla$-filtration dimension and the $\Delta$-filtration dimension of a non-zero object $M$ of $\mathcal{C}$ can easily be determined from the minimal tilting complex $C_\mathrm{min}(M)$.
We will need the following preliminary lemma.

\begin{Lemma} \label{lem:minimaltiltingcomplexnogap}
	Let $M$ be an object of $\mathcal{C}$ and write $C_\mathrm{min}(M) = (T_\bullet,d_\bullet)$.
	For any $i \geq 0$ such that $d_i = 0$, we have $T_j = 0$ for all $j > i$.
\end{Lemma}
\begin{proof}
	It is straightforward to verify that the complex $C$ given by
	\[ \qquad \cdots \xrightarrow{ d_{i-3} } T_{i-2} \xrightarrow{ d_{i-2} } T_{i-1} \xrightarrow{ d_{i-1} } T_i  \longrightarrow 0 \longrightarrow 0 \longrightarrow \cdots , \]
	 is minimal and has cohomology objects $H^0(M) \cong M$ and $H^k(M)=0$ for $k \neq 0$.
	 Thus $C \cong C_\mathrm{min}(M)$ in $C^b\big( \Tilt(\mathcal{C}) \big)$ and it follows that $T_j = 0$ for all $j>i$, as claimed.
\end{proof}

\begin{Lemma} \label{lem:minimaltiltingcomplexgfdwfd}
	Let $M$ be a non-zero object of $\mathcal{C}$ and write $C_\mathrm{min}(M) = (T_\bullet,d_\bullet)$.
	Then
	\[ \dim_\nabla(M) = \max\{ i \mid T_i \neq 0 \} \qquad \text{and} \qquad \dim_\Delta(M) = - \min\{ i \mid T_i \neq 0 \} . \]
\end{Lemma}
\begin{proof}
	We only show the first equality; the second one can be proven analogously.
	By the definition of $C_\mathrm{min}(M)$, we have $M \cong H^0\big( C_\mathrm{min}(M) \big)$ and there is a short exact sequence
	\[ 0 \longrightarrow \im(d_{-1}) \longrightarrow \ker(d_0) \longrightarrow M \longrightarrow 0 . \]
	We first show that $\ker(d_0)$ has $\nabla$-filtration dimension $r \coloneqq \max\{ i \mid T_i \neq 0 \}$.
	If $r=0$ then $\ker(d_0) = T_0$ is a tilting object and $\dim_\nabla(T_0) = 0$ by Remark \ref{rem:gfdzero}, so now suppose that $r>0$.
	For any $i \geq 0$, there is a short exact sequence
	\[ 0 \longrightarrow \ker(d_i) \longrightarrow T_i \longrightarrow \ker(d_{i+1}) \longrightarrow 0 , \]
	and by Lemma \ref{lem:minimaltiltingcomplexnogap}, we have $\ker(d_{i+1}) = \im(d_i) \neq 0$ for all $i < r$.
	Using part (1) of Lemma \ref{lem:GFDSES}, it follows that
	\[ \dim_\nabla\bigl( \ker(d_i) \bigr) \leq \max\big\{ \dim_\nabla(T_i) , \dim_\nabla\bigl( \ker(d_{i+1}) \bigr) + 1 \big\} = \dim_\nabla\bigl( \ker(d_{i+1}) \bigr) + 1 \]
	for $0 \leq i < r$, with equality if $0 = \dim_\nabla(T_i) \neq \mathrm{dim}_\nabla\bigl( \ker(d_{i+1}) \bigr)$.
	Now the claim will follow by induction once we show that $\dim_\nabla\big( \ker(d_{r-1}) \big) = 1$.
	As $\ker(d_r) = T_r$ is a tilting object of $\mathcal{C}$, we have $\dim_\nabla\bigl( \ker(d_r) \bigr) = 0$ and therefore $\dim_\nabla\bigl( \ker(d_{r-1}) \bigr) \leq 1$, and by applying part (4) of Lemma \ref{lem:GFDSES} to the short exact sequence
	\[ 0 \longrightarrow \ker(d_{r-1}) \longrightarrow T_{r-1} \longrightarrow T_r \longrightarrow 0 , \]
	we see that $\dim_\Delta\bigl( \ker(d_{r-1}) \bigr) = 0$, hence $\ker(d_{r-1})$ admits a $\Delta$-filtration.
	Furthermore, it is straightforward to see that the minimal tilting complex of $\ker(d_{r-1})$ is given by
	\[ C_\mathrm{min}\big( \ker(d_{r-1}) \big) \quad = \quad \big( \quad 0 \longrightarrow T_{r-1} \xrightarrow{d_{r-1}} T_r \longrightarrow 0 \quad \big) , \]
	with the term $T_{r-1}$ in homological degree zero, so $\ker(d_{r-1})$ is not a tilting object by Remark \ref{rem:minimaltiltingcomplexoftilting}.
	We conclude that $\ker(d_{r-1})$ does not admit a $\nabla$-filtration, hence $0 < \dim_\nabla\bigl( \ker(d_{r-1}) \bigr) \leq 1$, and it follows that $\dim_\nabla\bigl( \ker(d_{r-1}) \bigr) = 1$, as required.
	
	If $\im(d_{-1}) = 0$ then $M \cong \ker(d_0)$ and therefore $\dim_\nabla(M) = \dim_\nabla\bigl( \ker(d_0) \bigr) = r$, as claimed, so now suppose that $\im(d_{-1})$ is non-zero.
	By part (3) of Lemma \ref{lem:GFDSES}, we have
	\[ \dim_\nabla(M) \leq \max\big\{ \dim_\nabla\bigl( \im(d_1) \bigr) - 1 , \dim_\nabla\bigl( \ker(d_0) \bigr) \big\} = \max\big\{ \dim_\nabla\bigl( \im(d_{-1}) \bigr) - 1 , r \big\} \]
	with equality if $\dim_\nabla\bigl( \im(d_{-1}) \bigr) \neq r$, so it now suffices to prove that $\dim_\nabla\bigl( \im(d_1) \bigr) = 0$.
	For all $i<0$, we have a short exact sequence
	\[ 0 \longrightarrow \im(d_{i-1}) \longrightarrow T_i \longrightarrow \im(d_i) \longrightarrow 0 , \]
	and we can use part (3) of Lemma \ref{lem:GFDSES} and induction to see that $\dim_\nabla\bigl( \im(d_i) \bigr) = 0$ for all $i<0$ such that $\im(d_i) \neq 0$, as required.
\end{proof}

\begin{Corollary} \label{cor:Deltafiltrationtiltingresolution}
	An object $M$ of $\mathcal{C}$ admits a $\Delta$-filtration if and only if $C_\mathrm{min}(M)_i = 0$ for all $i<0$.
\end{Corollary}
\begin{proof}
	If $M=0$ then $M$ admits a $\Delta$-filtration by definition and $C_\mathrm{min}(M) = 0$.
	Otherwise $M$ admits a $\Delta$-filtration if and only if
	\[ 0 = \dim_\Delta( M ) = - \min\big\{ i \mathrel{\big|} C_\mathrm{min}(M)_i \neq 0 \big\} \]
	by Remark \ref{rem:gfdzero} and Lemma \ref{lem:minimaltiltingcomplexgfdwfd}, from which the claim is immediate.
\end{proof}

A similar result to Lemma \ref{lem:minimaltiltingcomplexgfdwfd} above appears in Lemma 6 of \cite{MazorchukOvsienkoFinistic}.
In the case where $\mathcal{C}$ is the category of finite-dimensional rational representations of a reductive algebraic group, Corollary \ref{cor:Deltafiltrationtiltingresolution} is also an immediate consequence of the main result of \cite{DonkinFiniteResolutions}.

We conclude this section by showing that minimal tilting complexes are well-behaved with respect to distinguished triangles in $D^b(\mathcal{C})$.
Note that for any short exact sequence $0 \to A \to B \to C \to 0$ in $\mathcal{C}$, there is a distinguished triangle $A \to B \to C \to A[1]$ in $D^b(\mathcal{C})$.
Therefore, the lemma can also be used to relate the minimal tilting complexes of any three objects of $\mathcal{C}$ that fit into a short exact sequence.

\begin{Lemma} \label{lem:triangle}
	For any distinguished triangle $X \to Y \to Z \to X[1]$ in $D^b(\mathcal{C})$, we have
	\[ C_\mathrm{min}(Z)_i \overset{\oplus}{\subseteq} C_\mathrm{min}(X)_{i+1} \oplus C_\mathrm{min}(Y)_i \]
	for all $i \in \Z$, and for $\lambda \in \Lambda$, we further have 
	\begin{align*}
	\big[ C_\mathrm{min}(X)_{i+1} : T_\lambda \big]_\oplus & + \big[ C_\mathrm{min}(Y)_i : T_\lambda \big]_\oplus
	\geq \big[ C_\mathrm{min}(Z)_i : T_\lambda \big]_\oplus \\
	& \geq  \big[ C_\mathrm{min}(X)_{i+1} : T_\lambda \big]_\oplus - \big[ C_\mathrm{min}(X)_i : T_\lambda \big]_\oplus - \big[ C_\mathrm{min}(X)_{i+2} : T_\lambda \big]_\oplus \\
	& \hspace{2cm} + \big[ C_\mathrm{min}(Y)_i : T_\lambda \big]_\oplus - \big[ C_\mathrm{min}(Y)_{i-1} : T_\lambda \big]_\oplus - \big[ C_\mathrm{min}(Y)_{i+1} : T_\lambda \big]_\oplus .
	\end{align*}
\end{Lemma}
\begin{proof}
	By the definition of minimal tilting complexes, we have $X \cong C_\mathrm{min}(X)$ in $D^b(\mathcal{C})$, and similarly for $Y$ and $Z$, whence there is a distinguished triangle
	\[ C_\mathrm{min}(X) \to C_\mathrm{min}(Y) \to C_\mathrm{min}(Z) \to C_\mathrm{min}(X)[1] \]
	in the derived category $D^b(\mathcal{C})$.
	This implies that $C_\mathrm{min}(Z)$ is isomorphic in $D^b( \mathcal{C} )$ to the mapping cone $C = \mathrm{cone}(f)$ of a homomorphism $f \colon C_\mathrm{min}(X) \to C_\mathrm{min}(Y)$, and by Lemma \ref{lem:minimalcomplex}, it follows that $C_\mathrm{min}(Z)$ is the minimal complex of $C$ and that there is a split monomorphism $C_\mathrm{min}(Z) \to C$ in the category of complexes $C^b\big( \Tilt(\mathcal{C}) \big)$.
	Since the term of $C$ in degree $i \in \Z$ is given by $C_i = C_\mathrm{min}(X)_{i+1} \oplus C_\mathrm{min}(Y)_i$, we conclude that
	\[ C_\mathrm{min}(Z)_i \overset{\oplus}{\subseteq} C_i = C_\mathrm{min}(X)_{i+1} \oplus C_\mathrm{min}(Y)_i , \]
	and Corollary \ref{cor:minimalcomplexdirectsummandbound} further yields
	\[ \big[ C_i : T_\lambda \big]_\oplus \geq \big[ C_\mathrm{min}(Z)_i : T_\lambda \big]_\oplus \geq \big[ C_i : T_\lambda \big]_\oplus - \big[ C_{i-1} : T_\lambda \big]_\oplus - \big[ C_{i+1} : T_\lambda \big]_\oplus \]
	for all $\lambda \in \Lambda$, from which the inequality in the lemma is immediate.
\end{proof}

\section{Standard objects and simple objects} \label{sec:standardsimple}

Let us keep the notation and assumptions from the preceding section; in particular $\mathcal{C}$ is a highest weight category with weight poset $(\Lambda,\leq)$.
In this section, we study the minimal tilting complexes of standard objects and simple objects in $\mathcal{C}$, starting with the standard objects.
The following result has already been observed by V.\ Mazorchuk and S.\ Ovsienko, albeit without giving a detailed proof;
see the remark after Corollary 3 in \cite{MazorchukOvsienkoPairing}.
Recall that for a finite order ideal $\Gamma \subseteq \Lambda$, we write $\mathcal{C}_\Gamma^\prime$ for the Ringel dual of the truncated subcategory $\mathcal{C}_\Gamma$ of $\mathcal{C}$.

\begin{Lemma} \label{lem:minimaltiltingcomplexstandard}
	Let $\lambda \in \Lambda$ and write $C_\mathrm{min}(\Delta_\lambda) = (T_\bullet,d_\bullet)$.
	For any $i \in \Z$ and $\mu \in \Lambda$ and for any finite order ideal $\Gamma \subseteq \Lambda$ with $\lambda \in \Gamma$, we have
	\[ [ T_i : T_\mu ]_\oplus = \begin{cases} \dim \Ext_{\mathcal{C}_\Gamma^\prime}^i(L_\mu^\prime,\nabla_\lambda^\prime) & \text{if } i \geq 0 \text{ and } \mu \leq \lambda , \\ 0 & \text{otherwise} . \end{cases} \]
\end{Lemma}
\begin{proof}
	Recall from Lemma \ref{lem:minimaltiltingcomplexorderideal} that all terms of $C_\mathrm{min}(\Delta_\lambda)$ belong to $\mathcal{C}_\Gamma$, so we may assume that $\Lambda = \Gamma$ is finite.
	(This argument also shows that $[ T_i : T_\mu ]_\oplus = 0$ unless $\mu \leq \lambda$, by setting $\Gamma=\{ \nu \in \Lambda \mid \nu \leq \lambda \}$.)
	By Corollary \ref{cor:Deltafiltrationtiltingresolution}, we have $C_\mathrm{min}(\Delta_\lambda)_i = 0$ for all $i<0$, whence $C_\mathrm{min}(\Delta_\lambda)$ affords a minimal tilting coresolution
	\[ 0 \to \Delta_\lambda \to T_0 \to T_1 \to T_2 \to \cdots \]
	of $\Delta_\lambda$ in $\mathcal{C}$.
	By applying the Ringel duality functor $R_\mathcal{C}^\nabla$, we obtain a minimal injective coresolution
	\begin{equation} \label{eq:minimalinjcores}
	0 \to \nabla_\lambda^\prime \xrightarrow{ \, ~ \, } I_0 \xrightarrow{ \, d_0^\prime \, } I_1 \xrightarrow{ \, d_1^\prime \, } I_2 \xrightarrow{ \, d_2^\prime \, } \cdots
	\end{equation}
	of $\nabla_\lambda^\prime$ in the Ringel dual $\mathcal{C}^\prime$ with $[T_i:T_\mu]_\oplus = [I_i:I_\mu^\prime]_\oplus = \dim \Hom_{\mathcal{C}^\prime}(L^\prime_\mu,I_i)$ by Theorem \ref{thm:Ringeldual}.
	Now the $\Ext$-group $\Ext_{\mathcal{C}^\prime}^i(L^\prime_\mu,\nabla^\prime_\lambda)$ is isomorphic to the $i$-th cohomology group of the complex
	\[ \Hom_{\mathcal{C}^\prime}( L^\prime_\mu , I_0 ) \to \Hom_{\mathcal{C}^\prime}( L^\prime_\mu , I_1 ) \to \Hom_{\mathcal{C}^\prime}( L^\prime_\mu , I_2 ) \to \cdots . \]
	All differentials of the latter complex are zero by minimality of the injective coresolution \eqref{eq:minimalinjcores}, cf.\ the proof of Corollary 2.5.4 in \cite{Benson}, and we obtain 
	\[ \dim \Ext_{\mathcal{C}^\prime}^i( L^\prime_\mu , \nabla^\prime_\lambda ) = \dim \Hom_{\mathcal{C}^\prime}( L^\prime_\mu , I_i ) = [ I_i : I^\prime_\mu ]_\oplus = [ T_i : T_\mu ]_\oplus , \]
	as required.
\end{proof}

Now we start our investigation of minimal tilting complexes of simple objects in $\mathcal{C}$.
In order to prove our main result, we will need the following technical proposition,
where for objects $A$ and $B$ of $\mathcal{C}$, we adopt the notation
\[ \dim \Ext_\mathcal{C}^\bullet(A,B) = \sum_{i \geq 0} \dim \Ext_\mathcal{C}^i(A,B) \cdot v^i \in \Z[v] \]
and for bounded complexes $X$ and $Y$ in $\mathcal{C}$, we write
\[ \dim \Hom_{D^b(\mathcal{C})}^\bullet(X,Y) = \sum_{i \in \Z} \dim \Hom_{D^b(\mathcal{C})}( X , Y[i] ) \cdot v^i \in \Z[v^{\pm 1}] . \]

\begin{Proposition} \label{prop:trianglestandardresolution}
	Let $X$ be an object of $\mathcal{C}$.
	Then there exist a sequence $n_0 , \ldots , n_r$ of non-negative integers, a sequence $\lambda_1, \ldots, \lambda_r$ of weights and a sequence $X_0 , \ldots , X_{r+1}$ of bounded complexes in $\mathcal{C}$, with $X_0 \cong X$ in $D^b(\mathcal{C})$ and $X_{r+1} = 0$,
	such that there are distinguished triangles
	\[ \Delta_{\lambda_i}[n_i] \longrightarrow X_i \longrightarrow X_{i+1} \longrightarrow \Delta_{\lambda_i}[n_i+1] \]
	for $i=0,\ldots,r$ and we have
	\[ \dim \Ext_\mathcal{C}^\bullet( X , \nabla_\lambda ) = \sum_{i=0}^r \delta_{\lambda,\lambda_i} \cdot v^{n_i} \]
	for all $\lambda \in \Lambda$.
\end{Proposition}
\begin{proof}
	We construct the three sequences inductively, in such a way that
	\[ \dim \Hom_{D^b(\mathcal{C})}^\bullet\big( X_i , \nabla_\lambda \big) = \dim \Hom_{D^b(\mathcal{C})}^\bullet\big( X_{i+1} , \nabla_\lambda \big) + \delta_{\lambda,\lambda_i} \cdot v^{n_i} \]
	for all $\lambda \in \Lambda$ and $i=0,\ldots,r$.
	Suppose that the object $X_i$ has already been constructed.
	If $H^\bullet(X_i) = 0$ then $X_i \cong 0$ in $D^b(\mathcal{C})$, and the claim follows with $r+1 = i$ since
	\[ \dim \Ext_\mathcal{C}^\bullet( X , \nabla_\lambda ) = \dim \Hom_{D^b(\mathcal{C})}^\bullet( X_0 , \nabla_\lambda )
	= \dim \Hom_{D^b(\mathcal{C})}^\bullet( X_{r+1} , \nabla_\lambda ) + \sum_{i=0}^r \delta_{\lambda,\lambda_i} \cdot v^{n_i} = \sum_{i=0}^r \delta_{\lambda,\lambda_i} \cdot v^{n_i} \]
	for all $\lambda \in \Lambda$, as required.
	Now suppose that $H^\bullet(X_i) \neq 0$ and denote by $\Gamma \subseteq \Lambda$ the smallest order ideal containing the highest weights of all composition factors of $\bigoplus_{j \in \Z} H^j(X_i)$.
	By Proposition \ref{prop:derivedtruncated}, the canonical functor $\mathcal{C}_\Gamma \to \mathcal{C}$ induces an equivalence between $D^b(\mathcal{C}_\Gamma)$ and the full subcategory $D^b_{\mathcal{C}_\Gamma}(\mathcal{C})$  of $D^b(\mathcal{C})$ whose objects are the complexes with cohomology in $\mathcal{C}_\Gamma$; hence we may assume that $X_i$ is represented by a complex with terms in $\mathcal{C}_\Gamma$.
	For a maximal element $\lambda \in \Gamma$, the simple object $L_\lambda$ is a composition factor of $H^m(X_i)$ for some $m \in \Z$, the standard object $\Delta_\lambda$ is projective in $\mathcal{C}_\Gamma$ and the costandard object $\nabla_\lambda$ is injective in $\mathcal{C}_\Gamma$.
	In particular, the degree zero cohomology functor $H^0(-)$ induces isomorphisms
	\[ \Hom_{D^b(\mathcal{C})}\big( \Delta_\lambda[-m] , X_i \big) \cong \Hom_\mathcal{C}\big( \Delta_\lambda , H^m(X_i) \big) \]
	and
	\[ \Hom_{D^b(\mathcal{C})}\big( X_i , \nabla_\lambda[-m] \big) \cong \Hom_\mathcal{C}\big( H^m(X_i) , \nabla_\lambda \big) . \]
	Again by the maximality of $\lambda$, there are homomorphisms $f \colon \Delta_\lambda \to H^m(X_i)$ and $g \colon H^m(X_i) \to \nabla_\lambda$ such that $g \circ f$ is non-zero (see Lemma \ref{lem:standardcostandardcomposition}), and we write $\hat{f}$ and $\hat{g}$ for their preimages under the above isomorphisms.
	We define $X_{i+1}$ as the cone of the homomorphism $\hat{f} \colon \Delta_\lambda[-m] \to X_i$, so that there is a distinguished triangle
	\begin{equation} \label{eq:triangleDeltalambdaXiXi+1}
	\Delta_\lambda[-m] \xrightarrow{~\hat{f}~} X_i \longrightarrow X_{i+1} \longrightarrow \Delta_\lambda[-m+1] .
	\end{equation}
	For all $\mu \in \Lambda$, this distinguished triangle gives rise to an exact sequence
	\[
	\cdots \to \Hom_{D^b(\mathcal{C})}\big( X_{i+1}[j] , \nabla_\mu \big) \to \Hom_{D^b(\mathcal{C})}\big( X_i[j] , \nabla_\mu \big) \to \Hom_{D^b(\mathcal{C})}\big( \Delta_\lambda[j-m] ,\nabla_\mu \big) \to \cdots , 
	\]
	and using the $\Ext$-vanishing property \eqref{eq:extvanishing}, it is straightforward to see that our choice of $\hat{f}$ implies that
	\[ \dim \Hom_{D^b(\mathcal{C})}^\bullet\big( X_i , \nabla_\mu \big) = \dim \Hom_{D^b(\mathcal{C})}^\bullet\big( X_{i+1} , \nabla_\mu \big) + \delta_{\lambda,\mu} \cdot v^{-m} . \]
	Thus, the claim is satisfied with $n_i = -m$ and $\lambda_i = \lambda$.
	
	By applying the functor $H^0(-)$ to the distinguished triangle \eqref{eq:triangleDeltalambdaXiXi+1}, we obtain an exact sequence
	\[ 0 \longrightarrow H^{m-1}( X_i ) \longrightarrow H^{m-1}(X_{i+1}) \longrightarrow \Delta_\lambda \xrightarrow{~f~} H^m( X_i ) \longrightarrow H^m( X_{i+1} )  \longrightarrow 0 \]
	and isomorphisms $H^j( X_i ) \cong H^j( X_{i+1} )$ for $j \notin \{ m-1 , m \}$.
	As $f$ is non-zero and $L_\lambda$ is the unique simple quotient of $\Delta_\lambda$, we conclude that the composition multiplicity of $L_\lambda$ in $\bigoplus_{j \in \Z} H^j( X_{i+1} )$ is strictly smaller than the composition multiplicity of $L_\lambda$ in $\bigoplus_{j \in \Z} H^j( X_i )$; hence by induction, we eventually arrive at a complex $X_{r+1}$ that is isomorphic to zero in $D^b(\mathcal{C})$.
\end{proof}

Although we were not able to find an explicit formula for the multiplicities of indecomposable tilting objects in the terms of minimal tilting complexes of simple objects in $\mathcal{C}$, we can establish an upper bound for these multiplicities as follows:

\begin{Proposition} \label{prop:minimaltiltingcomplexsimple}
	For $\lambda \in \Lambda$, there exists a bounded complex $C_\lambda = (T_\bullet,d_\bullet)$ of tilting objects in $\mathcal{C}$ with 
	\[ L_\lambda \cong C_\lambda \qquad \text{in } D^b(\mathcal{C}) \]
	and such that the terms of $C_\lambda$ satisfy $[ T_i : T_\mu ]_\oplus = 0$ for all $i \in \Z$ and $\mu \in \Lambda$ with $\mu \nleq \lambda$ and
	\[ [ T_i : T_\mu ]_\oplus = \sum_{\nu \leq \lambda} ~ \sum_{k-j=i} \dim \Ext_\mathcal{C}^j(L_\lambda,\nabla_\nu) \cdot \dim \Ext_{\mathcal{C}_\Gamma^\prime}^k(L_\mu^\prime,\nabla_\nu^\prime) \]
	for all $i \in \Z$ and $\mu \in \Lambda$ with $\mu \leq \lambda$ and for any finite order ideal $\Gamma \subseteq \Lambda$ with $\lambda \in \Gamma$.
	Furthermore, the minimal tilting complex $C_\mathrm{min}(L_\lambda)$ is a direct summand of $C_\lambda$ in the category of complexes in $\Tilt(\mathcal{C})$.
\end{Proposition}
\begin{proof}
	By Proposition \ref{prop:trianglestandardresolution}, there are sequences $n_0,\ldots,n_r$ of non-negative integers, $\lambda_0 , \ldots , \lambda_r$ of weights and $X_0 , \ldots , X_{r+1}$ of bounded complexes in $\mathcal{C}$, with $X_0 \cong L_\lambda$ in $D^b(\mathcal{C})$ and $X_{r+1} = 0$, such that there exist distinguished triangles
	\[ \Delta_{\lambda_i}[n_i] \longrightarrow X_i \longrightarrow X_{i+1} \longrightarrow \Delta_{\lambda_i}[n_i+1] \]
	for $i=0,\ldots,r$ and we have
	\begin{equation} \label{eq:extsimplecostandard}
	\dim \Ext_\mathcal{C}^\bullet\big( L_\lambda , \nabla_\mu \big) = \sum_{i=0}^r \delta_{\lambda_i,\mu} \cdot v^{n_i} .
	\end{equation}
	By triangle rotation, the object $X_i$ is isomorphic to the cone of a homomorphism $X_{i+1}[-1] \to \Delta_{\lambda_i}[n_i]$ for $i=0,\ldots,r$, and by descending induction on $i$, we see that $L_\lambda \cong X_0$ is isomorphic in $D^b(\mathcal{C})$ to a bounded complex $C_\lambda = ( T_\bullet , d_\bullet )$ of tilting objects with terms
	\[ T_j \coloneqq \bigoplus_{i=0}^r C_\mathrm{min}( \Delta_{\lambda_i}[n_i] )_j = \bigoplus_{i=0}^r C_\mathrm{min}( \Delta_{\lambda_i} )_{j+n_i} . \]
	By equation \eqref{eq:extsimplecostandard}, we have $\Ext_\mathcal{C}^{n_i}( L_\lambda , \nabla_{\lambda_i} ) \neq 0$ for $i=0,\ldots,r$ and therefore $\lambda \geq \lambda_i$ by Lemma \ref{lem:Extvanishingsimplecostandard}.
	For any finite order ideal $\Gamma \subseteq \Lambda$ with $\lambda \in \Gamma$, we further have
	\[ \big[ C_\mathrm{min}( \Delta_{\lambda_i} )_{j+n_i} : T_\mu \big]_\oplus = \dim \Ext_{\mathcal{C}_\Gamma^\prime}^{j+n_i}( L^\prime_\mu , \nabla^\prime_{\lambda_i} ) \]
	for all $\mu \in \Lambda$, $j \in \Z$ and $i \in \{ 0 , \ldots, r \}$ by Lemma \ref{lem:minimaltiltingcomplexstandard}, and for $\nu \in \Lambda$ and $m \geq 0$, we have
	\[ \dim \Ext_\mathcal{C}^m( L_\lambda , \nabla_\nu ) = \abs{ \{ i \mid \lambda_i = \nu \text{ and } n_i = m \} } \]
	by equation \eqref{eq:extsimplecostandard}.
	We conclude that
	\begin{align*}
	[ T_j : T_\mu ]_\oplus & = \sum_{i=0}^r \dim \Ext_{\mathcal{C}_\Gamma^\prime}^{j+n_i}( L^\prime_\mu , \nabla^\prime_{\lambda_i} ) \\
	& = \sum_{\nu \leq \lambda} \; \sum_{k \geq 0} \dim \Ext_{\mathcal{C}_\Gamma^\prime}^k( L^\prime_\mu , \nabla^\prime_\nu ) \cdot \abs{ \{ i \mid \lambda_i = \nu \text{ and } j+n_i = k \} } \\
	& = \sum_{\nu \leq \lambda} \; \sum_{k \geq 0} \dim \Ext_{\mathcal{C}_\Gamma^\prime}^k( L^\prime_\mu , \nabla^\prime_\nu ) \cdot \dim \Ext_\mathcal{C}^{k-j}( L_\lambda , \nabla_\nu )
	\end{align*}
	for all $j \in \Z$, as required.
	The minimal tilting complex $C_\mathrm{min}(L_\lambda)$ is a direct summand of $C_\lambda$ in the category of complexes in $\Tilt(\mathcal{C})$ by Lemma \ref{lem:minimaltiltingcomplex}.
\end{proof}

If the $\Ext$-groups that we considered in the preceding proposition satisfy certain parity vanishing conditions then our upper bound for the multiplicities of indecomposable tilting objects appearing in the terms of minimal tilting complexes of simple objects in $\mathcal{C}$ is sharp.

\begin{Corollary} \label{cor:minimaltiltingcomplexsimple}
	Let $\lambda \in \Lambda$ and let $\Gamma \subseteq \Lambda$ be a finite order ideal containing $\lambda$.
	For $\mu , \nu \in \Gamma$, consider the polynomials
	\[ p_{\nu,\lambda} = \sum_{i \geq 0} \dim \Ext_\mathcal{C}^i( L_\lambda , \nabla_\nu ) \cdot v^i \qquad \text{and} \qquad p^\prime_{\nu,\mu} = \sum_{i \geq 0} \dim \Ext_{\mathcal{C}_\Gamma^\prime}^i( L^\prime_\mu , \nabla^\prime_\nu ) \cdot v^i . \]
	Suppose that there is a function $\ell\colon \Gamma \to \Z$ such that
	\[ p_{\nu,\lambda} \in v^{\ell(\lambda)-\ell(\nu)} \cdot \Z[v^{\pm 2}] \qquad \text{and} \qquad p^\prime_{\nu,\mu} \in v^{\ell(\mu)-\ell(\nu)} \cdot \Z[v^{\pm 2}] \]
	for all $\nu \in \Gamma$.
	Then
	\[ \sum_{i \in \Z} \big[ C_\mathrm{min}( L_\lambda )_i : T_\mu \big]_\oplus \cdot v^i = \sum_{\nu \leq \lambda} \overline{ p_{\nu,\lambda} } \cdot p^\prime_{\nu,\mu} , \]
	where $\overline{\phantom{A}}\colon \Z[v^{\pm 1}] \to \Z[v^{\pm 1}]$ denotes the unique ring homomorphism with $v \mapsto v^{-1}$.
\end{Corollary}
\begin{proof}
	By Proposition \ref{prop:minimaltiltingcomplexsimple}, there is a bounded complex $C_\lambda=(T_\bullet,d_\bullet)$ in $\Tilt(\mathcal{C})$ such that the simple object $L_\lambda$ is isomorphic to $C_\lambda$ in $D^b(\mathcal{C})$ and
	\[ \sum_{i\in\Z} [ T_i : T_\mu ]_\oplus \cdot v^i = \sum_{\nu \leq \lambda} \overline{ p_{\nu,\lambda} } \cdot p^\prime_{\nu,\mu} \eqqcolon \mathfrak{t}_{\mu,\lambda} . \]
	Thus we have $C_\lambda \cong L_\lambda \cong C_\mathrm{min}(L_\lambda)$ in $D^b(\mathcal{C})$, so $C_\mathrm{min}(L_\lambda)$ is the minimal tilting complex of $C_\lambda$ by Lemma \ref{lem:minimaltiltingcomplex}, and using Corollary \ref{cor:minimalcomplexdirectsummandbound}, it follows that
	\[ [ T_i : T_\mu ]_\oplus \geq [ C_\mathrm{min}(L_\lambda)_i : T_\mu ]_\oplus \geq [ T_i : T_\mu ]_\oplus - [ T_{i-1} : T_\mu ]_\oplus - [ T_{i+1} : T_\mu ]_\oplus \]
	for all $i \in \Z$.
	Next observe that by assumption, we have
	\[ \overline{ p_{\nu,\lambda} } \cdot p^\prime_{\nu,\mu} \in v^{\ell(\lambda)-\ell(\nu)} \cdot v^{\ell(\mu)-\ell(\nu)} \cdot \Z[v^{\pm 2}] = v^{\ell(\lambda)+\ell(\mu)} \cdot \Z[v^{\pm 2}] \]
	for all $\nu \in \Gamma$, so that $\mathfrak{t}_{\mu,\lambda} \in v^{\ell(\lambda)+\ell(\mu)} \cdot \Z[v^{\pm 2}]$.
	In particular, we have $[T_i:T_\mu]_\oplus = 0$ for all $i \in \Z$ such that $i \not\equiv \ell(\lambda) + \ell(\mu) \mod 2$.
	It follows that $[ C_\mathrm{min}(L_\lambda)_i : T_\mu ]_\oplus = [ T_i : T_\mu ]_\oplus$ for all $i \in \Z$ and the claim is immediate.
\end{proof}

\begin{Remark}
	In the terminology of \cite[Definition 3.3]{CPSabstractKL}, the assumption that
	\[ p_{\nu,\lambda} \in v^{\ell(\lambda)-\ell(\nu)} \cdot \Z[v^{\pm 2}] \qquad \text{ for all } \nu \in \Gamma , \]
	as in Corollary \ref{cor:minimaltiltingcomplexsimple}, means that the simple object $L_\lambda$ has a \emph{left parity} in $\mathcal{C}_\Gamma$ with respect to the \emph{length function} $\ell \colon \Gamma \to \Z$.
	Similarly, the assumption that
	\[ p^\prime_{\nu,\mu} \in v^{\ell(\mu)-\ell(\nu)} \cdot \Z[v^{\pm 2}] \qquad \text{ for all } \nu \in \Gamma \]
	means that $L^\prime_\mu$ has a left parity in $\mathcal{C}_\Gamma^\prime$ with respect to $\ell \colon \Gamma \to \Z$. 
\end{Remark}

\section{Examples} \label{sec:examples}

In this section, we determine the multiplicities of indecomposable tilting objects appearing in the terms of minimal tilting complexes of standard objects and simple objects in certain categories of representations of complex semisimple Lie algebras, affine Kac-Moody algebras and quantum groups at roots of unity.

\subsection{Complex simple Lie algebras} \label{subsec:categoryO}

Let $\mathfrak{g}$ be a complex semisimple Lie algebra with a Borel subalgebra $\mathfrak{b}$ and a Cartan subalgebra $\mathfrak{h} \subseteq \mathfrak{b}$, and let $\Phi \subseteq \mathfrak{h}^*$ be the root system of $\mathfrak{g}$ with respect to $\mathfrak{h}$.
We write $\Phi^+ \subseteq \Phi$ for the positive system corresponding to $\mathfrak{b}$ and let $\Pi \subseteq \Phi^+$ be a base of $\Phi$, and we consider the partial order $\leq$ on $\mathfrak{h}^*$ with
\[ \lambda \leq \mu \qquad \text{if and only if} \qquad \mu - \lambda \in \sum_{\alpha \in \Pi} \Z_{\geq 0} \cdot \alpha \]
for $\lambda,\mu \in \mathfrak{h}^*$.
Let $( - \, , - )$ be an invariant bilinear form on $\mathfrak{g}$ and recall that there is an isomorphism
\[ \mathfrak{h}^* \longrightarrow \mathfrak{h}, \qquad \lambda \longmapsto h_\lambda \]
such that $\lambda(x) = ( h_\lambda , x )$ for all $x \in \mathfrak{h}$ and $\lambda \in \mathfrak{h}^*$.
We further suppose that $(h_\beta,h_\beta) = 2$ for every long root $\beta \in \Phi$, and we define the \emph{coroot} of a root $\alpha \in \Phi$ as $\alpha^\vee = \frac{2 h_\alpha}{(h_\alpha,h_\alpha)}$.
The reflection $s_\alpha \in \GL(\mathfrak{h}^*)$ corresponding to $\alpha$ is defined by $s_\alpha(\lambda) = \lambda - \lambda(\alpha^\vee) \cdot \alpha$ for $\lambda \in \mathfrak{h}^*$, and we write $W = \langle s_\alpha \mid \alpha \in \Phi \rangle$ for the Weyl group of $\Phi$.
Then $W$ is a finite Coxeter group with simple reflections $S = \{ s_\alpha \mid \alpha \in \Pi \}$, and we denote by $w_0$ the longest element of $W$.
We let
\[ X = \{ \lambda \in \mathfrak{h}^* \mid \lambda(\alpha^\vee) \in \Z \text{ for all } \alpha \in \Phi \} \qquad \text{and} \qquad X^+ = \{ \lambda \in X \mid \lambda(\alpha^\vee) \geq 0 \text{ for all } \alpha \in \Phi^+ \} \]
be the sets of \emph{integral weights} and \emph{dominant weights}, respectively, and write $\rho = \frac12 \cdot \sum_{\alpha \in \Phi^+} \alpha$ for the half-sum of all positive roots.
The \emph{dot action} of $W$ on $\mathfrak{h}^*$ is defined by
\[ w \Cdot \lambda = w( \lambda + \rho ) - \rho \]
for $w \in W$ and $\lambda \in \mathfrak{h}^*$.

We consider the category $\mathcal{O}$ of finitely generated $\mathfrak{g}$-modules that are locally $\mathfrak{b}$-finite and admit a weight space decomposition with respect to $\mathfrak{h}$; see for instance \cite{HumphreysCategoryO}.
For all $\lambda \in \mathfrak{h}^*$, the Verma module $\Delta_\lambda$ of highest weight $\lambda$ and its unique simple quotient $L_\lambda$ are objects of $\mathcal{O}$, and so is the dual Verma module $\nabla_\lambda$.
For every weight $\lambda \in - \rho - X^+ = \{ - \rho - \lambda \mid \lambda \in X^+ \}$, we denote by $\mathcal{O}_\lambda$ the block of $\mathcal{O}$ containing the simple $\mathfrak{g}$-module $L_\lambda$.
The stabilizer of $\lambda$ with respect to the dot action of $W$ on $\mathfrak{h}^*$ is a parabolic subgroup of $W$, and for $I \subseteq S$ such that $\Stab_W(\lambda) = W_I = \langle I \rangle$, the isomorphism classes of simple $\mathfrak{g}$-modules in $\mathcal{O}_\lambda$ are in bijection with the set $W^I$ of elements $w \in W$ that have minimal length in the coset $w W_I$.

Now let $J \subseteq S$ and denote by $\mathfrak{p}_J$ the parabolic subalgebra of $\mathfrak{g}$ that is generated by $\mathfrak{b}$ together with the root spaces $\mathfrak{g}_{-\alpha}$ for $\alpha \in \Pi$ with $s_\alpha \in J$.
We write $\mathcal{O}^{\mathfrak{p}_J}$ for the full subcategory of $\mathcal{O}$ whose objects are the locally $\mathfrak{p}_J$-finite $\mathfrak{g}$-modules in $\mathcal{O}$, and for $\lambda \in -\rho - X^+$, let $\mathcal{O}_\lambda^{\mathfrak{p}_J}$ be the full subcategory of $\mathcal{O}_\lambda$ whose objects are the locally $\mathfrak{p}_J$-finite $\mathfrak{g}$-modules in $\mathcal{O}_\lambda$.
For $I \subseteq S$ such that $\Stab_W(\lambda) = W_I$ and for $w \in W^I$, the simple $\mathfrak{g}$-module $L_{w\Cdot\lambda}$ belongs to $\mathcal{O}_\lambda^{\mathfrak{p}_J}$ if and only if $(w\Cdot\lambda)(\alpha^\vee) \geq 0$ for all $\alpha \in J$, or equivalently, if $w \in w_J \prescript{J}{}{W} \cap W^I = w_J \prescript{J}{}{W}^I_\mathrm{reg}$.
Here, we write $w_J$ for the longest element of $W_J$ and $\prescript{J}{}{W}$ for the set of elements $w \in W$ that have minimal length in the coset $W_J w$.
Furthermore, we denote by $\prescript{J}{}{W}^I_\mathrm{reg}$ the set of minimal length representatives for the regular double cosets in $W_J \backslash W / W_I$, where a double coset $W_J w W_I$, with $w \in \prescript{J}{}{W} \cap W^I$, is called regular if $J \cap w I w^{-1} = \varnothing$.
(See also Appendix A in \cite{GruberSingularParabolic}.)
In particular, the isomorphism classes of simple $\mathfrak{g}$-modules in $\mathcal{O}_\lambda^{\mathfrak{p}_J}$ are in bijection with the set $w_J \prescript{J}{}{W}^I_\mathrm{reg}$ via $w \mapsto L_{w\Cdot\lambda}$.
For $\mu \in \mathfrak{h}^*$, let us write $\Delta^{\mathfrak{p}_J}_\mu$ for the largest quotient of $\Delta_\mu$ that belongs to $\mathcal{O}^{\mathfrak{p}_J}$ and $\nabla^{\mathfrak{p}_J}_\mu$ for the largest submodule of $\nabla_\mu$ that belongs to $\mathcal{O}^{\mathfrak{p}_J}$, and call these $\mathfrak{g}$-modules the \emph{generalized} or \emph{parabolic} Verma module and dual Verma module, respectively, of highest weight $\mu$ with respect to $\mathfrak{p}_J$.
Then $\mathcal{O}_\lambda^{\mathfrak{p}_J}$ is a highest weight category with weight poset $(w_J \prescript{J}{}{W}^I_\mathrm{reg},\leq)$, where $\leq$ denotes the Bruhat order, with simple objects $L_{w\Cdot\lambda}$, standard objects $\Delta^{\mathfrak{p}_J}_{w\Cdot\lambda}$ and costandard objects $\nabla^{\mathfrak{p}_J}_{w\Cdot\lambda}$ for $w \in w_J \prescript{J}{}{W}^I_\mathrm{reg}$.

Now for $x,y \in W$, let us write $h_{x,y}$ and $h^{x,y}$ for the corresponding Kazhdan-Lusztig polynomial and inverse Kazhdan-Lusztig polynomial, respectively, in the normalization of \cite{SoergelKL}.
For a subset $I \subseteq S$ and $x,y \in \prescript{I}{}{W}$, we denote by $m^I_{x,y}$, $m_I^{x,y}$, $n^I_{x,y}$ and $n_I^{x,y}$ the corresponding spherical and anti-spherical (inverse) parabolic Kazhdan-Lusztig polynomials, as in \cite[Section 3]{SoergelKL}.
For $J \subseteq S$, $\lambda \in - \rho - X^+$ with $\Stab_W(\lambda) = W_I$ and $x \in w_J \prescript{J}{}{W}^I_\mathrm{reg}$, let us further write $T^{\mathfrak{p}_J}_{x\Cdot\lambda}$ for the indecomposable tilting $\mathfrak{g}$-module of highest weight $x\Cdot\lambda$ in $\mathcal{O}^{\mathfrak{p}_J}_\lambda$.
We can now determinate the multiplicities of indecomposable tilting $\mathfrak{g}$-modules appearing in the terms of minimal tilting complexes of generalized Verma modules and simple $\mathfrak{g}$-modules.

\begin{Theorem} \label{thm:minimaltiltingcomplexescategoryO}
	Let $I,J \subseteq S$ and let $\lambda \in - \rho - X^+$ such that $\Stab_W(\lambda) = W_I$.
	Then, for all $i \in \Z$ and all $x,y \in w_J \prescript{J}{}{W}^I_\mathrm{reg}$, we have
	\[ \sum_{i \in \Z} \big[ C_\mathrm{min}( \Delta^{\mathfrak{p}_J}_{x\Cdot\lambda} )_i : T^{\mathfrak{p}_J}_{y\Cdot\lambda} \big]_\oplus \cdot v^i = n^I_{w_I y^{-1} w_J w_0,w_I x^{-1} w_J w_0} = m_I^{ x^{-1} w_J , y^{-1} w_J } \]
	and
	\[ \sum_{i \in \Z} \big[ C_\mathrm{min}( L_{x\Cdot\lambda} )_i : T^{\mathfrak{p}_J}_{y\Cdot\lambda} \big]_\oplus \cdot v^i = \sum_{z \in w_J \prescript{J}{}{W}^I_\mathrm{reg}} \overline{ n^I_{z^{-1},x^{-1}} } \cdot m_I^{ z^{-1} w_J , y^{-1} w_J } \]
	where we take minimal tilting complexes in the highest weight category $\mathcal{O}^{\mathfrak{p}_J}_\lambda$.
\end{Theorem}
\begin{proof}
	By Theorem 2.2 in \cite{GruberSingularParabolic}, we have
	\[ \sum_{i \geq 0} \dim \Ext_{\mathcal{O}^{\mathfrak{p}_J}}^i\big( L_{x\Cdot\lambda} , \nabla^{\mathfrak{p}_J}_{y\Cdot\lambda} \big) \cdot v^i = n^I_{y^{-1},x^{-1}} . \]
	Furthermore, the Ringel dual $(\mathcal{O}_\lambda^{\mathfrak{p}_J})^\prime$ of $\mathcal{O}_\lambda^{\mathfrak{p}_J}$ is equivalent to $\mathcal{O}_\lambda^{\mathfrak{p}_{\widehat{J}}}$ by Theorem 8.1 in \cite{CoulembierMazorchukDualities}, where we write $\widehat{J} = w_0 J w_0$, and the Ringel duality functor
	\[ R_{\mathcal{O}_\lambda^{\mathfrak{p}_J}}^\nabla \colon \mathcal{O}_\lambda^{\mathfrak{p}_J} \longrightarrow (\mathcal{O}_\lambda^{\mathfrak{p}_J})^\prime \simeq \mathcal{O}_\lambda^{\mathfrak{p}_{\widehat{J}}} \]
	satisfies
	\[ R_{\mathcal{O}_\lambda^{\mathfrak{p}_J}}^\nabla\big( \Delta^{\mathfrak{p}_J}_{ x \Cdot \lambda } \big) \cong \nabla^{\mathfrak{p}_{\widehat{J}}}_{ w_0 w_J x w_I \Cdot \lambda } \]
	for all $x \in w_J \prescript{J}{}{W}^I_\mathrm{reg}$.
	Thus, the dimensions of the $\Ext$-groups between the simple object corresponding to $y$ and the costandard object corresponding to $x$ in the Ringel dual $( \mathcal{O}^{\mathfrak{p}_J}_\lambda )^\prime$ are given by
	\[ \sum_{i \geq 0} \dim \Ext_{\mathcal{O}^{\mathfrak{p}_{\widehat{J}}}_\lambda}^i( L_{w_0w_Jyw_I} , \nabla^{\mathfrak{p}_J}_{w_0w_Jxw_I\Cdot\lambda} ) \cdot v^i = n^I_{w_Ix^{-1}w_Jw_0,w_Iy^{-1}w_Jw_0} , \]
	and by Proposition 3.9 in \cite{SoergelKL}, we have $n^I_{w_Ix^{-1}w_Jw_0,w_Iy^{-1}w_Jw_0} = m_I^{x^{-1}w_J,y^{-1}w_J}$.
	Now the first claim is immediate from Lemma \ref{lem:minimaltiltingcomplexstandard}.
	The second claim follows from Corollary \ref{cor:minimaltiltingcomplexsimple} because
	\[ n^I_{z^{-1},x^{-1}} \in v^{\ell(x) - \ell(z)} \cdot \Z[v^{\pm2}] \qquad \text{and} \qquad m_I^{z^{-1}w_J,y^{-1}w_J} \in v^{\ell(y) - \ell(z)} \cdot \Z[v^{\pm2}] \]
	for all $x,y,z \in w_J \prescript{J}{}{W}^I_\mathrm{reg}$ by Remark 3.2 in \cite{SoergelKL}.
\end{proof}

\begin{Remark} \label{rem:minimaltiltingcomplexescategoryOspecialcases}
	If we set $I = J = \varnothing$ in Theorem \ref{thm:minimaltiltingcomplexescategoryO} then we obtain
	\[ \sum_{i \in \Z} \big[ C_\mathrm{min}( \Delta_{x\Cdot\lambda} )_i : T_{y\Cdot\lambda} \big]_\oplus \cdot v^i = h^{x,y} \qquad \text{and} \qquad \sum_{i \in \Z} \big[ C_\mathrm{min}( L_{x\Cdot\lambda} )_i : T_{y\Cdot\lambda} \big]_\oplus \cdot v^i = \sum_{z \in W} \overline{ h_{z,x} } \cdot h^{z,y} . \]
\end{Remark}

\begin{Remark} \label{rem:nonintegralcategoryO}
	Although we have only discussed minimal tilting complexes for simple $\mathfrak{g}$-modules and Verma modules of integral highest weight, one can get similar results for $\mathfrak{g}$-modules with non-integral highest weight using W.\ Soergel's results from \cite{SoergelBlocksCategoryO}.
\end{Remark}

\subsection{Affine Kac-Moody algebras} \label{subsec:KacMoody}

Let us keep the notation from the previous subsection and additionally assume that $\mathfrak{g}$ is a complex \emph{simple} Lie algebra.
Recall that we fix a Borel subalgebra $\mathfrak{b}$ and a Cartan subalgebra $\mathfrak{h} \subseteq \mathfrak{b}$ of $\mathfrak{g}$.
We consider the affine Kac-Moody Lie algebra
\[ \mathfrak{\tilde g} = ( \C[t^{\pm1}] \otimes_\C \mathfrak{g} ) \oplus \C c \oplus \C d \]
as in \cite[Section 6]{TanisakiCharacterFormulas} and its subalgebras
\[ \mathfrak{\tilde b} = \mathfrak{b} \oplus ( t \C[t] \otimes \mathfrak{g} ) \oplus \C c \oplus \C d \qquad \text{and} \qquad \mathfrak{\tilde h} = \mathfrak{h} \oplus \C c \oplus \C d , \]
and we view $\mathfrak{h}^*$ as a subspace of $\mathfrak{\tilde h}^*$, with the convention that $\lambda(c) = \lambda(d) = 0$ for all $\lambda \in \mathfrak{h}^*$.
The root system $\tilde \Phi$ of $\mathfrak{\tilde g}$ is the disjoint union of the sets $\tilde \Phi_\mathrm{re}$ of real roots and $\tilde \Phi_\mathrm{im}$ of imaginary roots.
We denote by $\tilde \Phi^+$ the set of positive roots corresponding to $\mathfrak{\tilde b}$, by $\tilde \Phi_\mathrm{re}^+ = \tilde \Phi_\mathrm{re} \cap \tilde \Phi^+$ the set of positive real roots and by $\tilde \Pi = \Pi \sqcup \{ \alpha_0 \}$ the set of simple roots, as in \cite[Section 6]{TanisakiCharacterFormulas} or \cite[Section 13.1]{KumarKacMoody}.
The invariant bilinear form $( - \, , - )$ on $\mathfrak{g}$ can be extended to an invariant symmetric bilinear form on $\mathfrak{\tilde g}$ (see Lemma 13.1.8 in \cite{KumarKacMoody}),
and as before, we can define the coroot $\alpha^\vee$ corresponding to a root $\alpha \in \tilde \Phi_\mathrm{re}$ and the reflection $s_\alpha \in \GL(\mathfrak{\tilde h}^*)$ with $s_\alpha(\lambda) = \lambda - \lambda( \alpha^\vee ) \cdot \alpha$ for all $\lambda \in \mathfrak{\tilde h}^*$.
The \emph{(affine) Weyl group} of $\mathfrak{\tilde g}$ is the subgroup $\tilde W = \langle s_\alpha \mid \alpha \in \tilde \Pi \rangle$ of $\GL(\mathfrak{\tilde h}^*)$; it is a Coxeter group with set of simple reflections $\tilde S = \{ s_\alpha \mid \alpha \in \tilde \Pi \}$ and it is canonically isomorphic to the semidirect product $W \rtimes \Z\Phi^\vee$, where $W$ denotes the Weyl group of $\mathfrak{g}$ and $\Z\Phi^\vee$ is the coroot lattice of $\mathfrak{g}$ (see Proposition 13.1.7 in \cite{KumarKacMoody}).
We define $\tilde \rho \in \mathfrak{\tilde h}^*$ by $\tilde \rho(\alpha^\vee) = 1$ for all $\alpha \in \tilde \Pi$ and $\tilde \rho(d) = 0$ and consider the \emph{dot action} of $\tilde W$ on $\mathfrak{\tilde h}^*$ given by
\[ w \Cdot \lambda = w( \lambda + \tilde \rho ) - \tilde \rho \]
for all $w \in \tilde W$ and $\lambda \in \mathfrak{\tilde h}^*$.
We say that $\lambda \in \mathfrak{\tilde h}^*$ has \emph{level} $e \in \C$ if $e = ( \lambda + \tilde \rho )(c)$ and we set
\[ \lambda^\prime \coloneqq - \lambda - 2 \tilde \rho , \]
so that $\lambda^\prime$ has level $-e$ when $\lambda$ has level $e$ and $(w \Cdot \lambda)^\prime = w\Cdot\lambda^\prime$ for all $w \in \tilde W$.
The sets of \emph{integral} and \emph{dominant} weights are defined by
\[ \tilde X = \{ \lambda \in \mathfrak{\tilde h}^* \mid \lambda(\alpha^\vee) \in \Z \text{ for all } \alpha \in \tilde \Pi \} \qquad \text{and} \qquad \tilde X^+ = \{ \lambda \in \tilde X \mid \lambda(\alpha^\vee) \geq 0 \text{ for all } \alpha \in \tilde \Pi \} , \]
respectively, and we note that for every integral weight $\mu \in \tilde X$ of positive (or negative) level, there is a unique weight $\nu \in \tilde X^+ \setminus\{0\}$ such that $\mu$ belongs to the $\tilde W$-orbit of $\nu - \tilde \rho$ (respectively $-\nu - \tilde \rho$) with respect to the dot action; see for instance Section 6.3 in \cite{BrundanStroppel}.

Now let us consider the category $\mathcal{\tilde O}$ of $\mathfrak{\tilde g}$-modules that admit a weight space decomposition with finite-dimensional weight spaces with respect to $\mathfrak{\tilde h}$ and whose set of weights is contained in the union of finitely many sets of the form $\{ \mu \in \mathfrak{\tilde h}^* \mid \mu \leq \lambda \}$ for $\lambda \in \mathfrak{\tilde h}^*$, where we write $\leq$ for the partial order on $\mathfrak{\tilde h}^*$ that is defined by
\[ \mu \leq \lambda \qquad \text{if and only if} \qquad \lambda - \mu \in \sum_{\alpha \in \tilde\Pi} \Z_{\geq 0} \cdot \alpha . \]
For all $\lambda \in \mathfrak{\tilde h}^*$, the Verma module $\Delta_\lambda$ of highest weight $\lambda$, the dual Verma module $\nabla_\lambda$ and the unique simple quotient $L_\lambda$ of $\Delta_\lambda$ (which is also the unique simple submodule of $\nabla_\lambda$) are objects of $\mathcal{\tilde O}$.
Even though a $\mathfrak{\tilde g}$-module $M$ in $\mathcal{\tilde O}$ does not need to admit a finite composition series, the composition multiplicities $[ M : L_\mu ]$ are still (well-defined and) finite for $\mu \in \mathfrak{\tilde h}^*$ (see Lemma 2.1.9 in \cite{KumarKacMoody}).
Now let us fix
\[ \lambda \in ( - \tilde \rho + \tilde X^+ ) \cup ( - \tilde \rho - \tilde X^+ ) \qquad \text{with} \qquad \lambda \neq - \tilde \rho \]
and write $\mathcal{\tilde O}_\lambda$ for the full subcategory of $\mathcal{\tilde O}$ whose objects are the $\mathfrak{\tilde g}$-modules $M$ in $\mathcal{\tilde O}$ such that $[ M : L_\mu ] = 0$ for all $\mu \in \mathfrak{\tilde h}^* \setminus \tilde W \Cdot \lambda$.
Furthermore, for $J \subseteq \tilde S$, we write $\mathfrak{p}_J$ for the Levi subalgebra of $\mathfrak{\tilde g}$ corresponding to $J$ and we let $\mathcal{\tilde O}^{\mathfrak{p}_J}$ (or $\mathcal{\tilde O}^{\mathfrak{p}_J}_\lambda$) be the full subcategory of $\mathcal{\tilde O}$ (or $\mathcal{\tilde O}_\lambda$) whose objects are the locally $\mathfrak{p}_J$-finite $\mathfrak{\tilde g}$-modules in $\mathcal{\tilde O}$ (or $\mathcal{\tilde O}_\lambda$).
As in the previous subsection, we write $\Delta^{\mathfrak{p}_J}_\mu$ for the largest locally $\mathfrak{p}_J$-finite quotient of the Verma module $\Delta_\mu$ of highest weight $\mu \in \mathfrak{\tilde h}^*$ and $\nabla^{\mathfrak{p}_J}_\mu$ for the largest locally $\mathfrak{p}_J$-finite submodule of $\nabla_\mu$, and we call these $\mathfrak{\tilde g}$-modules the \emph{generalized Verma module} and the \emph{generalized dual Verma module} of highest weight $\mu$ with respect to $\mathfrak{p}_J$.

Now let $I \subseteq \tilde S$ such that $\Stab_{\tilde W}(\lambda) = \tilde W_I$.
If $\lambda$ has positive (or negative) level then we parameterize the set of isomorphism classes of simple $\mathfrak{\tilde g}$-modules in $\mathcal{\tilde O}_\lambda$ by $X_\lambda \coloneqq \tilde W^I w_I$ (respectively $X_\lambda \coloneqq \tilde W^I$) via the map $w \mapsto L_{w\Cdot\lambda}$, and this parametrization restricts to a bijection between the set of isomorphism classes of simple $\mathfrak{\tilde g}$-modules in $\mathcal{\tilde O}_\lambda^{\mathfrak{p}_J}$ and the set $X_\lambda^J \coloneqq \prescript{J}{}{\tilde W}^I_\mathrm{reg} w_I$ (respectively $X_\lambda^J \coloneqq w_J\prescript{J}{}{\tilde W}^I_\mathrm{reg}$) of representatives for the regular double cosets in $\tilde W_J \backslash \tilde W / \tilde W_I$.
If $\lambda$ has negative level then $\mathcal{\tilde O}_\lambda$ is a highest weight category in the sense of Definition \ref{def:HWC} according to Theorem 6.4 in \cite{BrundanStroppel}, but if $\lambda$ has positive level then the $\mathfrak{\tilde g}$-modules in $\mathcal{\tilde O}_\lambda$ generally do not have finite composition series.
(The category $\mathcal{\tilde O}_\lambda$ is an \emph{upper finite highest weight category} in the sense of \cite[Definition 3.34]{BrundanStroppel} when $\lambda$ has positive level, see again \cite[Theorem 6.4]{BrundanStroppel}.)
In order to overcome this problem, we consider truncations of the category $\mathcal{\tilde O}_\lambda^{\mathfrak{p}_J}$ as follows:
If $\lambda$ has negative level then for $w \in X_\lambda^J$, we write $\mathcal{\tilde O}^{\mathfrak{p}_J}_{\lambda,\leq w}$ for the Serre subcategory of $\mathcal{\tilde O}^{\mathfrak{p}_J}_{\lambda}$ generated by the simple $\mathfrak{\tilde g}$-modules $L_{x\Cdot\lambda}$ with $x \in X_\lambda^{J, \leq w}$, where we set
\[ X_\lambda^{J, \leq w} \coloneqq \{ y \in X_\lambda^J \mid y \leq w \} . \]
Then $\mathcal{\tilde O}^{\mathfrak{p}_J}_{\lambda,\leq w}$ is a highest weight category with weight poset $( X_\lambda^{J, \leq w} , \leq )$ by Lemma 3.7 in \cite{ShanVaragnoloVasserot}, where $\leq$ denotes the Bruhat order and the standard objects and costandard objects are given by $\Delta^{\mathfrak{p}_J}_{x\Cdot\lambda}$ and $\nabla^{\mathfrak{p}_J}_{x\Cdot\lambda}$, respectively, for $x \in X_\lambda^{J, \leq w}$.
If $\lambda$ has positive level then for $w \in X_\lambda^J$, we write $\mathcal{\tilde O}^{\mathfrak{p}_J}_{\lambda,\leq w}$ for the Serre quotient category of $\mathcal{\tilde O}^{\mathfrak{p}_J}_{\lambda}$ by the Serre subcategory generated by the simple $\mathfrak{\tilde g}$-modules $L_{x\Cdot\lambda}$ for $x \in X_\lambda^J$ with $x \nleq w$, so that the simple objects of $\mathcal{\tilde O}^{\mathfrak{p}_J}_{\lambda,\leq w}$ are naturally parameterized by the set
\[ X_\lambda^{J,\leq w} \coloneqq \{ y \in X_\lambda^J \mid y \leq w \} . \]
Then $\mathcal{\tilde O}^{\mathfrak{p}_J}_{\lambda,\leq w}$ is a highest weight category with weight poset $( X_\lambda^{J, \leq w} , \geq )$ by Proposition 3.9 in \cite{ShanVaragnoloVasserot}, where $\geq$ denotes the opposite Bruhat order and the standard objects and costandard objects are given by $\Delta^{\mathfrak{p}_J}_{x\Cdot\lambda}$ and $\nabla^{\mathfrak{p}_J}_{x\Cdot\lambda}$, respectively, for $x \in X_\lambda^{J, \leq w}$.

\begin{Remark}
	For $\lambda$ of negative level, the fact that all of the truncated categories $\mathcal{\tilde O}_{\lambda,\leq w}^{\mathfrak{p}_J}$ are highest weight categories implies that $\mathcal{\tilde O}_\lambda^{\mathfrak{p}_J}$ is itself a highest weight category with weight poset $(X_\lambda^J,\leq)$ by Definition 3.50 and Corollary 3.64 in \cite{BrundanStroppel}, where as before, we write $\leq$ for the Bruhat order.
	For $\lambda$ of positive level, the category $\mathcal{\tilde O}_\lambda^{\mathfrak{p}_J}$ is an \emph{upper finite highest weight category} with weight poset $(X_\lambda^J,\geq)$ in the sense of \cite[Definition 3.34]{BrundanStroppel} because every simple $\mathfrak{\tilde g}$-module in $\mathcal{\tilde O}_\lambda^{\mathfrak{p}_J}$ has a projective cover in $\mathcal{\tilde O}_\lambda^{\mathfrak{p}_J}$ which admits a filtration by generalized Verma modules; see \cite[Section 4]{RochaCaridiWallach}.
\end{Remark}

We are now ready to discuss minimal tilting complexes in the category $\mathcal{\tilde O}$.
Let us fix $\lambda \in - \tilde \rho - \tilde X^+$ with stabilizer $\Stab_{\tilde W}(\lambda) = \tilde W_I$ for $I \subsetneq \tilde S$,
and for $x \in X_\lambda = \tilde W^I$, let us write $T_{x\Cdot\lambda}$ for the indecomposable tilting $\mathfrak{\tilde g}$-module of highest weight $x \Cdot \lambda$.
For $J \subsetneq \tilde S$ and $x \in X_\lambda^J$, we further denote by $T^{\mathfrak{p}_J}_{x\Cdot\lambda}$ the indecomposable tilting $\mathfrak{\tilde g}$-module of highest weight $x\Cdot\lambda$ in $\mathcal{\tilde O}_\lambda^{\mathfrak{p}_J}$.

\begin{Theorem} \label{thm:minimaltiltingcomplexesKacMoody}
	Let $I,J \subsetneq \tilde S$ and let $\lambda \in - \tilde \rho - \tilde X^+$ such that $\Stab_{\tilde W}(\lambda) = \tilde W_I$.
	Then, for all $x,y \in X_\lambda^J$, we have
	\[ \sum_{i \in \Z} \big[ C_\mathrm{min}( \Delta^{\mathfrak{p}_J}_{x\Cdot\lambda} )_i : T^{\mathfrak{p}_J}_{y\Cdot\lambda} \big]_\oplus \cdot v^i = m_I^{x^{-1}w_J,y^{-1}w_J} \]
	and
	\[ \sum_{i \in \Z} \big[ C_\mathrm{min}( L_{x\Cdot\lambda} )_i : T^{\mathfrak{p}_J}_{y\Cdot\lambda} \big]_\oplus \cdot v^i = \sum_{z \in X_\lambda^J} \overline{ n^I_{z^{-1},x^{-1}} } \cdot m_I^{ z^{-1} w_J , y^{-1} w_J } \]
	where we take minimal tilting complexes in the highest weight category $\mathcal{\tilde O}^{\mathfrak{p}_J}_\lambda$.
\end{Theorem}
\begin{proof}
	Recall that we write $\lambda^\prime = - \lambda - 2 \tilde \rho$.
	By Theorem 3.3 in \cite{GruberSingularParabolic}, we have
	\[ \sum_{i \geq 0} \dim \Ext_{\mathcal{\tilde O}^{\mathfrak{p}_J}}^i( L_{y\Cdot\lambda} , \nabla^{\mathfrak{p}_J}_{x\Cdot\lambda} ) \cdot v^i = n^I_{x^{-1},y^{-1}} \]
	for all $x,y \in X_\lambda^J$ and
	\[ \sum_{i \geq 0} \dim \Ext_{\mathcal{\tilde O}^{\mathfrak{p}_J}}^i( L_{y\Cdot\lambda^\prime} , \nabla^{\mathfrak{p}_J}_{x\Cdot\lambda^\prime} ) \cdot v^i = m_I^{w_Ix^{-1},w_Iy^{-1}} \]
	for all $x,y \in X_{\lambda^\prime}^J$, and the same formulas hold in the truncated categories $\mathcal{\tilde O}^{\mathfrak{p}_J}_{\lambda,\leq w}$ and $\mathcal{\tilde O}^{\mathfrak{p}_J}_{\lambda^\prime,\leq w}$ by Theorems 3.42 and 3.59 in \cite{BrundanStroppel}, for $x,y \leq w$.
	Furthermore, the Ringel dual $(\mathcal{\tilde O}^{\mathfrak{p}_J}_{\lambda,\leq w})^\prime$ of $\mathcal{\tilde O}^{\mathfrak{p}_J}_{\lambda,\leq w}$ is equivalent to $\mathcal{\tilde O}^{\mathfrak{p}_J}_{\lambda^\prime,\leq w_Jww_I}$ for all $w \in X_\lambda^J$, and for $x \in X_\lambda^{J,\leq w}$, the Ringel duality functor
	\[ R_{\mathcal{\tilde O}_{\lambda,\leq w}^{\mathfrak{p}_J}}^\nabla \colon \mathcal{\tilde O}_{\lambda,\leq w}^{\mathfrak{p}_J} \longrightarrow ( \mathcal{\tilde O}_{\lambda,\leq w}^{\mathfrak{p}_J} )^\prime \simeq \mathcal{\tilde O}_{\lambda^\prime,\leq w_J w w_I}^{\mathfrak{p}_J} \]
	satisfies
	\[ R_{\mathcal{\tilde O}_{\lambda,\leq w}^{\mathfrak{p}_J}}^\nabla\big( \Delta^{\mathfrak{p}_J}_{x \Cdot\lambda} \big) \cong \nabla^{\mathfrak{p}_J}_{w_Jxw_I\Cdot\lambda^\prime} \]
	by Proposition 3.9 in \cite{ShanVaragnoloVasserot}.
	Thus the dimensions of the $\Ext$-groups between the simple object corresponding to $y$ and the costandard object corresponding to $x$ in $( \mathcal{\tilde O}^{\mathfrak{p}_J}_{\lambda,\leq w} )^\prime$  are given by
	\[ \sum_{i \geq 0} \dim \Ext_{ \mathcal{\tilde O}_{\lambda^\prime,\leq w_J w w_I}^{\mathfrak{p}_{J}} }^i\big( L_{w_Jyw_I\Cdot\lambda^\prime} , \nabla^{\mathfrak{p}_J}_{w_Jxw_I\Cdot\lambda^\prime} \big) \cdot v^i = m_I^{x^{-1}w_J,y^{-1}w_J} \]
	for $x,y \in X_\lambda^{J,\leq w}$, and the first claim is immediate from Lemma \ref{lem:minimaltiltingcomplexstandard}.
	The second claim follows from Corollary \ref{cor:minimaltiltingcomplexsimple} because
	\[ n^I_{z^{-1},x^{-1}} \in v^{\ell(x) - \ell(z)} \cdot \Z[v^{\pm2}] \qquad \text{and} \qquad m_I^{z^{-1}w_J,y^{-1}w_J} \in v^{\ell(y) - \ell(z)} \cdot \Z[v^{\pm2}] \]
	for $x,y,z \in X_\lambda^{J,\leq w}$ by Remark 3.2 in \cite{SoergelKL}.
\end{proof}

\begin{Remark}
	For $\lambda \in - \tilde \rho + \tilde X^+$ of \emph{positive} level and $I , J \subsetneq \tilde S$ such that $\Stab_{\tilde W}(\lambda) = W_I$, the category $\mathcal{\tilde O}_\lambda$ is generally not a highest weight category in the sense of Definition \ref{def:HWC} and may not contain any tilting $\mathfrak{\tilde g}$-modules.
	Nevertheless, we can compute the multiplicities of indecomposable tilting objects in minimal tilting complexes of simple objects and standard objects in the highest weight categories $\mathcal{\tilde O}_{\lambda,\leq w}^{\mathfrak{p}_J}$, for $w \in X_\lambda^J$.
	Specifically, writing $T_{y\Cdot\lambda}^{\mathfrak{p}_J,\leq w}$ for the indecomposable tilting object corresponding to $y \in X_\lambda^{J,\leq w}$ in $\mathcal{\tilde O}_{\lambda,\leq w}^{\mathfrak{p}_J}$, we obtain
	\[ \sum_{i \in \Z} \big[ C_\mathrm{min}( \Delta^{\mathfrak{p}_J}_{x\Cdot\lambda} )_i : T^{\mathfrak{p}_J, \leq w}_{y\Cdot\lambda} \big]_\oplus \cdot v^i = n^I_{w_Ix^{-1}w_J,w_Iy^{-1}w_J} \]
	and
	\[ \sum_{i \in \Z} \big[ C_\mathrm{min}( L_{x\Cdot\lambda} )_i : T^{\mathfrak{p}_J,\leq w}_{y\Cdot\lambda} \big]_\oplus \cdot v^i = \sum_{z \in X_\lambda^{J,\leq w}} \overline{ m_I^{w_Iz^{-1},w_Ix^{-1}} } \cdot n^I_{w_Ix^{-1}w_J,w_Iy^{-1}w_J} \]
	for all $x,y \in X_\lambda^{J,\leq w}$, where we take minimal tilting complexes in the highest weight category $\mathcal{\tilde O}_{\lambda,\leq w}^{\mathfrak{p}_J}$.
\end{Remark}

\begin{Remark} \label{rem:nonintegralKacMoody}
	From Theorem \ref{thm:minimaltiltingcomplexesKacMoody}, we can also obtain some information about minimal tilting complexes of simple $\mathfrak{\tilde g}$-modules and Verma modules of non-integral highest weight, using P.\ Fiebig's combinatorial description of blocks in BGG categories for Kac-Moody algebras \cite{FiebigCombinatoricsKacMoody}.
	More specifically, for $\lambda \in \mathfrak{h}^*$ with $(\lambda + \tilde \rho)(c) \in \Q_{<0}$, consider
	\[ \tilde \Phi_{\lambda,\mathrm{re}} = \{ \alpha \in \tilde \Phi_\mathrm{re} \mid \lambda(\alpha^\vee) \in \Z \} \qquad \text{and} \qquad \tilde \Phi_{\lambda,\mathrm{re}}^+ = \tilde \Phi_{\lambda,\mathrm{re}} \cap \tilde \Phi^+ \]
	and let $\tilde W_\lambda = \langle s_\alpha \mid \alpha \in \Phi_{\lambda,\mathrm{re}} \rangle$ be the integral Weyl group of $\lambda$.
	Then $\tilde W_\lambda$ is a Coxeter group whose simple reflections $\tilde S_\lambda$ are the $s_\alpha$ with $\alpha \in \tilde \Phi_{\lambda,\mathrm{re}}$ that permute the set $\tilde \Phi_{\lambda,\mathrm{re}}^+ \setminus \{ \alpha \}$.
	We write $\mathcal{\tilde O}_\lambda$ for the full subcategory of $\mathcal{\tilde O}$ whose objects are the $\mathfrak{\tilde g}$-modules in $\mathcal{\tilde O}$ all of whose composition factors have highest weight in $W_\lambda \Cdot \lambda$.
	If $\tilde W_\lambda$ is isomorphic to the Weyl group of some affine Kac-Moody Lie algebra then Theorem 11 in \cite{FiebigCombinatoricsKacMoody} implies that the multiplicities of indecomposable tilting $\mathfrak{\tilde g}$-modules appearing in the terms of minimal tilting complexes of simple $\mathfrak{\tilde g}$-modules and (generalized) Verma modules in $\mathcal{\tilde O}_\lambda$ (and its parabolic versions $\mathcal{\tilde O}_\lambda^{\mathfrak{p}_J}$) are governed by parabolic Kazhdan-Lusztig polynomials with respect to $\tilde W_\lambda$ and the parabolic subgroup $\Stab_{\tilde W}(\lambda)$, as in Theorem \ref{thm:minimaltiltingcomplexesKacMoody}.
\end{Remark}

\subsection{Quantum groups at roots of unity} \label{subsec:quantum}

We continue to use the notation from the two previous subsections; in particular $\mathfrak{g}$ is a complex simple Lie algebra and $\mathfrak{\tilde g}$ is the corresponding affine Kac-Moody algebra.
Let us fix $\ell>0$ and write $U_\zeta=U_\zeta(\mathfrak{g})$ for the quantum group corresponding to $\mathfrak{g}$ specialized at a primitive $\ell$-th root of unity $\zeta \in \C$, as defined for instance in Section 7 in \cite{TanisakiCharacterFormulas}, and let $\Rep(U_\zeta)$ be the category of finite-dimensional $U_\zeta$-modules that admit a weight space decomposition (again as in \cite{TanisakiCharacterFormulas}).
In the following, the objects of $\Rep(U_\zeta)$ will simply be referred to as $U_\zeta$-modules, and we write $\Hom_{U_\zeta}(-\,,-)$ and $\Ext_{U_\zeta}^i(-\,,-)$ for the $\Hom$-functors and the $\Ext$-functors in $\Rep(U_\zeta)$.
For $\lambda \in X^+$ (the set of dominant integral weights with respect to the root system $\Phi$ of $\mathfrak{g}$ and the base $\Pi \subseteq \Phi$), we write $\Delta^\zeta_\lambda$ for the Weyl module, $\nabla^\zeta_\lambda$ for the induced module and $L^\zeta_\lambda$ for the simple $U_\zeta$-module of highest weight $\lambda$ over $U_\zeta$, all as in Sections H.9--11 in \cite{Jantzen}.
Then for all $\lambda,\mu \in X^+$ and $i \in \Z$, we have
\[ \Ext_{U_\zeta}^i( \Delta^\zeta_\lambda , \nabla^\zeta_\mu ) \cong \begin{cases} \C & \text{if } \lambda = \mu \text{ and } i = 0 , \\ 0 & \text{otherwise} , \end{cases} \]
and it follows that $\Rep(U_\zeta)$ is a highest weight category with weight poset $(X^+,\leq)$, where as before, we write $\leq$ for the partial order on $X^+$ that is induced by the base $\Pi$.
A proof of the $\Ext$-vanishing property above can be found for instance in Theorem 3.1 in \cite{AndersenStroppelTubbenhauerAdditionalNotes}.
There, some restrictions are imposed on $\ell$, but these are not necessary, as remarked in \cite[Assumption 2.11]{AndersenStroppelTubbenhauerAdditionalNotes}.
See also Proposition 4.1 in \cite{AndersenLinkageQuantum} and Theorem 5.5 in \cite{RyomHansenQuantumKempf}.
Now for $\lambda \in X^+$, let us write $T^\zeta_\lambda$ for the indecomposable tilting $U_\zeta$-module of highest weight $\lambda$.
In order to determine the multiplicities of indecomposable tilting $U_\zeta$-modules in the minimal tilting complexes of simple $U_\zeta$-modules and of Weyl modules, we next discuss the decomposition of $\Rep(U_\zeta)$ into \emph{linkage classes}.

Recall that we have fixed an invariant bilinear form $(-\,,-)$ on $\mathfrak{g}$ such that $(h_\alpha,h_\alpha) = 2$ for any long root $\alpha \in \Phi$, where for $\lambda \in \mathfrak{h}^*$, we define $h_\lambda \in \mathfrak{h}$ by $\lambda = (h_\lambda,-)$.
For $\alpha \in \Phi$, let $\check \alpha \in \mathfrak{h}^*$ be the unique element in $\mathfrak{h}^*$ with $h_{\check \alpha} = \alpha^\vee$ and write $\check \Phi = \{ \check \alpha \mid \alpha \in \Phi \}$ for the preimage of the dual root system $\Phi^\vee$ of $\Phi$ under the isomorphism $\mathfrak{h}^* \to \mathfrak{h}$ with $\lambda \mapsto h_\lambda$.
Further let $D$ be the square of the ratio of the length of the highest root $\alpha_\mathrm{h} \in \Phi^+$ by the length of the highest short root $\alpha_\mathrm{hs} \in \Phi^+$ (i.e.\ the unique root in $\Phi$ whose coroot $\alpha_\mathrm{hs}^\vee$ is the highest root in $\Phi^\vee$), that is $D=1$ if $\Phi$ is simply-laced, $D=2$ if $\Phi$ is of type $\mathrm{B}_n$, $\mathrm{C}_n$ or $\mathrm{F}_4$ and $D=3$ if $\Phi$ is of type $\mathrm{G}_2$.
Then our conventions imply that we have $\check \alpha = \alpha$ for all long roots $\alpha \in \Phi$ and $\check \beta = D \beta$ for all short roots in $\beta \in \Phi$; in particular the root lattice $Q \coloneqq \Z \Phi$ and the coroot lattice $\check Q \coloneqq \Z \check \Phi$ satisfy $D Q \subseteq \check Q \subseteq Q$.
For $e \in \R$, we consider the $e$-dilated dot action of $Q \rtimes W$ on $X_\R \coloneqq X \otimes_\Z \R$ defined by
\[ t_\gamma w \Cdot_e \lambda = w( \lambda + \rho ) - \rho + e \gamma \]
for $\gamma \in Q$, $w \in W$ and $\lambda \in X_\R$, where we write $t_\gamma$ for the image of $\gamma$ under the canonical embedding of $Q$ into $Q \rtimes W$.
The $e$-dilated dot action of $\check Q \rtimes W$ on $\mathfrak{h}^*$ is defined analogously, and we will denote it by the same symbol $\Cdot_e$.
The closure of the fundamental dominant $e$-alcove
\[ A_e = \{ x \in X_\R \mid 0 \leq (x+\rho)(\alpha^\vee) \text{ for all } \alpha \in \Pi \text{ and } (x+\rho)(\alpha_\mathrm{hs}^\vee) \leq e \} \]
is a fundamental domain for the $e$-dilated dot action of $Q \rtimes W$ on $X_\R$, and similarly, the closure of
\[ \check A_e = \{ x \in X_\R \mid 0 \leq (x+\rho)(\alpha^\vee) \text{ for all } \alpha \in \Pi \text{ and } (x+\rho)(\alpha_\mathrm{h}^\vee) \leq e \} \]
is a fundamental domain for the $e$-dilated dot action of $\check Q \rtimes W$ on $X_\R$.

Now let us set
\[ \ell^\prime = \begin{cases} \ell & \text{if } \ell \text{ is odd} , \\ \ell / 2 & \text{if } \ell \text{ is even} \end{cases} \qquad \text{and} \qquad r_\ell = \begin{cases} \ell^\prime & \text{if } D \text{ does not divide } \ell^\prime , \\ \ell^\prime / D & \text{if } D \text{ divides } \ell^\prime . \end{cases} \]
We further define
	\[ W_\ell = \begin{cases} Q \rtimes W & \text{if } D \text{ does not divide } \ell^\prime , \\ \check Q \rtimes W & \text{if } D \text{ divides } \ell^\prime \end{cases} \qquad \text{and} \qquad C_\ell = \begin{cases} A_{r_\ell} & \text{if } D \text{ does not divide } \ell^\prime , \\ \check A_{r_\ell} & \text{if } D \text{ divides } \ell^\prime , \end{cases} \]
	and for $\lambda \in \overline{C}_\ell \cap X$, we write $\Rep_\lambda( U_\zeta )$ for the full subcategory of $\Rep(U_\zeta)$ whose objects are the $U_\zeta$-modules all of whose composition factors have highest weight in $W_\ell \Cdot_{r_\ell} \lambda$.
	Then, according to Corollary 4.4 in \cite{AndersenLinkageQuantum} (see also Sections 3.17--19 in \cite{AndersenParadowskiFusion}), every $U_\zeta$-module $M$ decomposes as a direct sum
	\[ M = \bigoplus_{\lambda \in \overline{C}_\ell \cap X} \pr_\lambda M , \]
	where $\pr_\lambda M$ denotes the largest $U_\zeta$-submodule of $M$ that belongs to $\Rep_\lambda(U_\zeta)$, for $\lambda \in \overline{C}_\ell \cap X$.%
	\footnote{Note that \cite{AndersenParadowskiFusion} works with the convention that short roots (and not long roots) have length $\sqrt{2}$.
	This is why in our notation, if $D$ divides $\ell^\prime$ then \cite{AndersenParadowskiFusion} considers the $\ell^\prime$-dilated dot action of $\frac{1}{D} \check Q \rtimes W$ on $X_\R$ and not the $r_\ell = \ell^\prime / D$-dilated dot action of $\check Q \rtimes W$.
	It is straightforward to see that the isomorphism $\frac{1}{D} \check Q \rtimes W \to \check Q \rtimes W$ with $t_\gamma \mapsto t_{D\gamma}$ for $\gamma \in \frac{1}{D} \check Q$ and $w \mapsto w$ for $w \in W$ intertwines these actions.}
	The group $W_\ell$ is a Coxeter group with simple reflections $S_\ell = S \cup \{ s_0 \}$, where $s_0 = t_{\alpha_{\mathrm{hs}}} s_{\alpha_{\mathrm{hs}}}$ if $D$ does not divide $\ell^\prime$ and $s_0 = t_{\check \alpha_\mathrm{h}} s_{\alpha_\mathrm{h}}$ if $D$ divides $\ell^\prime$.
	For $\lambda \in \overline{C}_\ell \cap X$, the stabilizer of $\lambda$ with respect to the $r_\ell$-dilated dot action of $W_\ell$ is a parabolic subgroup of $W_\ell$, and if $\Stab_{W_\ell}(\lambda) = \langle I \rangle = W_{\ell,I}$ for some $I \subseteq S_\ell$ then the orbit $W_\ell \Cdot_{r_\ell} \lambda$ is in bijection with the set $W_\ell^I$ of elements $x \in W_\ell$ that have minimal length in the coset $x W_{\ell,I}$, via $w \mapsto w \Cdot_{r_\ell} \lambda$.
	Observe that the subgroup $W$ of $W_\ell$ is also a parabolic subgroup of $W_\ell$, corresponding to the set of simple reflections $S \subseteq S_\ell$.
	For $x \in W_\ell^I$, it is straightforward to see that we have $x \Cdot_{r_\ell} \lambda \in X^+$ if and only if $(x^{-1} S x) \cap I = \varnothing$ and $x \in \prescript{S}{}{W}_\ell$ (i.e.\ $x$ has minimal length in the coset $W x$).
	This means that we have $x \Cdot_{r_\ell} \lambda \in X^+$ if and only if $x \in \prescript{S}{}{W}^I_{\ell,\mathrm{reg}}$, that is, the double coset $W x W_{\ell,I}$ is regular and $x$ has minimal length in $W x W_{\ell,I}$.
	Thus, the map $x \mapsto L_{x \Cdot \lambda}$ defines a bijection between $\prescript{S}{}{W}^I_{\ell,\mathrm{reg}}$ and the set of isomorphism classes of simple $U_\zeta$-modules in $\Rep_\lambda(G)$.

We are going to translate our results about minimal tilting complexes for representations of affine Kac-Moody Lie algebras from Subsection \ref{subsec:KacMoody} to the setting of quantum groups using the Kazhdan-Lusztig functor, as explained below.
Let us set $g = \tilde \rho(c)$, $k = - ( \ell / 2 D )- g$ and
\[ \mathfrak{\grave g} \coloneqq [ \mathfrak{\tilde g} , \mathfrak{\tilde g} ] = \C[t^{\pm 1}] \otimes \mathfrak{g} \oplus \C c . \]
In \cite{KL12,KL34}, D.\ Kazhdan and G.\ Lusztig define a category $\mathcal{\grave O}_k$ of $\mathfrak{\grave g}$-modules on which the central element $c$ acts by the scalar $k$,
and they construct a functor
\[ F_\ell \colon \mathcal{\grave O}_k \longrightarrow \Rep(U_\zeta) \]
which is an equivalence of categories when $\ell$ is sufficiently large.%
\footnote{If $\ell / 2D = r / m$ for coprime integers $r,m>0$ then the functor $F_\ell$ is an equivalence when $r \geq r_0$ for some bound $r_0$ depending on the root system $\Phi$.
For $\Phi$ of type $\mathrm{A}_n$ or $\mathrm{D}_{2n}$ (respectively $\mathrm{D}_{2n+1}$, $\mathrm{E}_6$, $\mathrm{E}_7$, $\mathrm{E}_8$),
we can choose $r_0 = 1$ (respectively $r_0 = 3 , 14 , 20 , 32$).
The bounds for the non-simply-laced root systems are not known explicitly (but see Remark 7.2 in \cite{TanisakiCharacterFormulas}).}
(See also \cite{LusztigMonodromic} for the non-simply-laced case and \cite[Section7]{TanisakiCharacterFormulas} for a summary of these results.)
We call $F_\ell$ the \emph{Kazhdan-Lusztig functor}.
Now suppose that $F_\ell$ is an equivalence, and for $\lambda \in \overline{C}_\ell \cap X$, let us set $\lambda^\ast = w_S \Cdot \lambda + k \chi \in \mathfrak{h}^*$.
Also consider the parabolic subalgebra $\mathfrak{p}_S$ of $\mathfrak{\tilde g}$ corresponding to $S \subseteq \tilde S$, so that $[\mathfrak{p}_S,\mathfrak{p}_S] \cong \mathfrak{g}$ and $\tilde W_S \cong W$.
Then, according to Proposition 5.5 and the discussion following Remark 7.2 in \cite{ParshallScottSemisimpleSeries}, the restriction to $\mathfrak{\grave g}$ of a $\mathfrak{\tilde g}$-module in $\mathcal{\tilde O}_{\lambda^\ast}^{\mathfrak{p}_S}$ belongs to the category $\mathcal{\grave O}_k$, and the functor $G \coloneqq F_\ell \circ \res_\mathfrak{\grave g}^\mathfrak{\tilde g}$ induces an equivalence between $\mathcal{\tilde O}_{\lambda^\ast}^{\mathfrak{p}_S}$ and $\Rep_\lambda(U_\zeta)$.
Furthermore, it follows from Sections 6 and 7 in \cite{TanisakiCharacterFormulas} that there is a canonical isomorphism $\varphi \colon W_\ell \to \tilde W_{\lambda^\ast}$ with $\varphi(S_\ell) = \tilde S_\lambda$, where $\tilde W_{\lambda^\ast}$ denotes the integral Weyl group of $\lambda^\ast$ as in Remark \ref{rem:nonintegralKacMoody}, such that
\[ G\big( L_{ w_S \varphi(x) \Cdot \lambda^\ast } \big) \cong L^\zeta_{x\Cdot_{r_\ell}\lambda} \qquad \text{and} \qquad G\big( \Delta^{\mathfrak{p}_S}_{ w_S \varphi(x) \Cdot \lambda^\ast } \big) \cong \Delta^\zeta_{x\Cdot_{r_\ell}\lambda} \]
for all $x \in \prescript{S}{}{W}^I_{\ell,\mathrm{reg}}$.

Now we are ready to state the main theorem of this subsection.

\begin{Theorem} \label{thm:minimaltiltingcomplexesquantum}
	Let $\ell > 0$ and suppose that the Kazhdan-Lusztig functor $F_\ell$ is an equivalence.
	Further let $\lambda \in \overline{C}_\ell \cap X$ and $I \subseteq S_\ell$ such that $\Stab_{W_\ell}(\lambda) = W_{\ell,I}$.
	Then we have
	\[ \sum_{i \in \Z} \big[ C_\mathrm{min}( \Delta^\zeta_{x \Cdot_{r_\ell} \lambda} )_i : T^\zeta_{y \Cdot_{r_\ell} \lambda} \big]_\oplus \cdot v^i = m_I^{x^{-1},y^{-1}} \]
	and
	\[ \sum_{i \in \Z} \big[ C_\mathrm{min}( L^\zeta_{x \Cdot_{r_\ell} \lambda} )_i : T^\zeta_{y \Cdot_{r_\ell} \lambda} \big]_\oplus \cdot v^i = \sum_{z \in \prescript{S}{}{W}^I_{\ell,\mathrm{reg}}} \overline{ n^I_{z^{-1}w_S,x^{-1}w_S} } \cdot m_I^{ z^{-1} , y^{-1} } \]
	for all $x,y \in \prescript{S}{}{W}^I_{\ell,\mathrm{reg}}$, where $w_S$ denotes the longest element of the parabolic subgroup $W$ of $W_\ell$.
\end{Theorem}
\begin{proof}
	This follows from Theorem \ref{thm:minimaltiltingcomplexesKacMoody} and Remark \ref{rem:nonintegralKacMoody} by applying the functor $G = F_\ell \circ \res_\mathfrak{\grave g}^\mathfrak{\tilde g}$.
\end{proof}

\begin{Remark} \label{rem:minimaltiltingcomplexesquantumspecialcases}
	Let us keep the notation and assumptions of Theorem \ref{thm:minimaltiltingcomplexesquantum}, but now additionally assume that $\lambda \in C_\ell \cap X$, so that $\Stab_{W_\ell}(\lambda) = \{ e \}$.
	(This implicitly forces that $r_\ell \geq \rho(\alpha_{\mathrm{hs}}^\vee)+1$ if $D$ does not divide $\ell^\prime$ and that $r_\ell \geq \rho(\alpha_{\mathrm{h}}^\vee)+1$ if $D$ divides $\ell^\prime$.)
	Then for all $x,y \in \prescript{S}{}{W}_\ell$, we have
	\begin{equation} \label{eq:minimaltiltingstandardquantum}
	\sum_{i \in \Z} \big[ C_\mathrm{min}( \Delta^\zeta_{x \Cdot_{r_\ell} \lambda} )_i : T^\zeta_{y \Cdot_{r_\ell} \lambda} \big]_\oplus \cdot v^i = h^{x^{-1},y^{-1}} = h^{x,y} = n_S^{x,y}
	\end{equation}
	by Theorem \ref{thm:minimaltiltingcomplexesquantum} and by Proposition 3.7 in \cite{SoergelKL}.
	We consider the formula \eqref{eq:minimaltiltingstandardquantum} as a homological lift of W.\ Soergel's character formula for indecomposable tilting $U_\zeta$-modules from \cite{SoergelCharakterformeln}, which asserts that $[ T^\zeta_{x\Cdot_{r_\ell}\lambda} : \Delta^\zeta_{y \Cdot_{r_\ell} \lambda} ]_\Delta = n^S_{y,x}(1)$ for all $x,y \in \prescript{S}{}{W}_\ell$.
\end{Remark}

\bibliographystyle{alpha}
\bibliography{tilting_v2}

\end{document}